\documentclass[12pt]{amsart}
\usepackage{amssymb}
\usepackage[shortlabels]{enumitem}
\usepackage[all]{xy}
\usepackage{xcolor}
\usepackage{mathrsfs}
\usepackage[margin=1in]{geometry} 
\usepackage[bookmarks, bookmarksdepth=2, colorlinks=true, linkcolor=blue, citecolor=blue, urlcolor=blue]{hyperref}
\usepackage{eucal}
\usepackage{contour}
\usepackage[normalem]{ulem}
\usepackage{dsfont}

\makeatletter
\@addtoreset{equation}{section}
\makeatother

\numberwithin{equation}{section}
\newtheorem{theorem}[equation]{Theorem}

\newtheorem{proposition}[equation]{Proposition}
\newtheorem{lemma}[equation]{Lemma}
\newtheorem{corollary}[equation]{Corollary}

\theoremstyle{definition}
\newtheorem{rmk}[equation]{Remark}
\newenvironment{remark}[1][]{\begin{rmk}[#1] \pushQED{\qed}}{\popQED \end{rmk}}
\newtheorem{eg}[equation]{Example}
\newenvironment{example}[1][]{\begin{eg}[#1] \pushQED{\qed}}{\popQED \end{eg}}
\newtheorem{defnaux}[equation]{Definition}
\newenvironment{definition}[1][]{\begin{defnaux}[#1]\pushQED{\qed}}{\popQED \end{defnaux}}

\newcommand{\bA}{\mathbf{A}}

\newcommand{\bC}{\mathbf{C}}
\newcommand{\cC}{\mathcal{C}}
\newcommand{\fC}{\mathfrak{C}}

\newcommand{\cD}{\mathcal{D}}
\newcommand{\fD}{\mathfrak{D}}

\newcommand{\bF}{\mathbf{F}}

\newcommand{\bG}{\mathbf{G}}

\newcommand{\bH}{\mathbf{H}}

\newcommand{\bN}{\mathbf{N}}

\newcommand{\bQ}{\mathbf{Q}}

\newcommand{\bR}{\mathbf{R}}

\newcommand{\fS}{\mathfrak{S}}

\newcommand{\fT}{\mathfrak{T}}

\newcommand{\bV}{\mathbf{V}}

\newcommand{\bZ}{\mathbf{Z}}

\newcommand{\arxiv}[1]{\href{http://arxiv.org/abs/#1}{{\tiny\tt arXiv:#1}}}

\newcommand{\DOI}[1]{\href{http://doi.org/#1}{\color{purple}{\tiny\tt DOI:#1}}}

\contourlength{1pt}

\contourlength{0.8pt}
\newcommand{\myuline}[1]{%
  \uline{\phantom{#1}}%
  \llap{\contour{white}{#1}}%
}
\DeclareMathOperator{\uRep}{\text{\myuline{\rm Rep}}}
\DeclareMathOperator{\uPerm}{\ul{Perm}}
\DeclareMathOperator{\udeg}{\text{\myuline{\rm deg}}}

\let\ul\underline

\let\wt\widetilde
\renewcommand{\phi}{\varphi}

\DeclareMathOperator{\End}{End}
\DeclareMathOperator{\Ind}{Ind}
\DeclareMathOperator{\Aut}{Aut}
\DeclareMathOperator{\Hom}{Hom}

\DeclareMathOperator{\Pic}{Pic}

\DeclareMathOperator{\Mod}{Mod}

\DeclareMathOperator{\Sym}{Sym}

\DeclareMathOperator{\uHom}{\ul{Hom}}
\DeclareMathOperator{\uEnd}{\ul{End}}
\DeclareMathOperator{\udim}{\ul{dim}}

\DeclareMathOperator{\Sp}{\mathbf{Sp}}
\DeclareMathOperator{\Br}{Br}
\DeclareMathOperator{\Rep}{Rep}
\newcommand{\id}{\mathrm{id}}

\newcommand{\GL}{\mathbf{GL}}
\newcommand{\PGL}{\mathbf{PGL}}

\newcommand{\defn}[1]{\emph{#1}}

\newcommand{\bone}{\mathbf{1}}
\newcommand{\op}{\mathrm{op}}
\newcommand{\uotimes}{\mathbin{\ul{\otimes}}}

\title[Interpolation and Azumaya algebras]{Interpolation of the oscillator representation \\ and Azumaya algebras in tensor categories}

\author{Andrew Snowden}
\address{Department of Mathematics, University of Michigan, Ann Arbor, MI, USA}
\email{\href{mailto:asnowden@umich.edu}{asnowden@umich.edu}}
\urladdr{\url{http://www-personal.umich.edu/~asnowden/}}
\thanks{The author was supported by NSF grant DMS-2301871.}

\date{July 31, 2024}

\begin{document}

\begin{abstract}
Let $\fC$ be a symmetric tensor category and let $A$ be an Azumaya algebra in $\fC$. Assuming a certain invariant $\eta(A) \in \Pic(\fC)[2]$ vanishes, and fixing a certain choice of signs, we show that there is a universal tensor functor $\Phi \colon \fC \to \fD$ for which $\Phi(A)$ splits. We apply this when $\fC=\uRep(\Sp_t(\bF_q))$ is the interpolation category of finite symplectic groups and $A$ is a certain twisted group algbera in $\fC$, and we show that the splitting category $\fD$ is S.\ Kriz's interpolation category of the oscillator representation. This construction has a number advantages over previous ones; e.g., it works in non-semisimple cases. It also brings some conceptual clarity to the situation: the existence of Kriz's category is tied to the non-triviality of the Brauer group of $\fC$.
\end{abstract}

\maketitle
\tableofcontents

\section{Introduction}

\subsection{Overview}

Pre-Tannakian categories are a natural class of tensor categories generalizing representation categories of algebraic groups. Constructing new examples is an important and challenging problem. The first examples\footnote{We ignore the Verlinde category and related examples in this discussion.} going beyond the (super) Tannakian world were the interpolation categories $\uRep(\GL_t)$ and $\uRep(\fS_t)$ of Deligne--Milne \cite{DeligneMilne} and Deligne \cite{Deligne}. Knop \cite{Knop1, Knop2} systematized the interpolation idea, and introduced important new examples such as $\uRep(\GL_t(\bF_q))$.

In joint work with Harman \cite{repst}, we defined a tensor category $\uPerm(G, \mu)$ associated to an oligomorphic group $G$ and a measure $\mu$, which sometimes admits a pre-Tannakian envelope. Our construction recovers all previously known interpolation categories of finite groups. It also leads to some fundamentally new examples, such as the Delannoy category \cite{line, circle}. Moreover, we show that every ``discrete'' pre-Tannakian category is obtained from our construction \cite{discrete}; discrete roughly means that there are enough \'etale algebras. The interpolation categories of algebraic groups, like $\uRep(\GL_t)$, are not discrete, and thus do not come from an oligomorphic group; however, we hypothesize that they come from an analogous theory for ``algebraic oligomorphic groups'' \cite{homoten}.

Recently, S.~Kriz \cite{Kriz} introduced a new and very intriguing pre-Tannakian category by interpolating the oscillator representation of finite symplectic groups. This category is not discrete (as she proves), but is very close to being so. It thus the first example we know that cannot come from the envisioned algebraic oligomorphic theory---at least, not directly.

This paper is an attempt to better understand Kriz's category, and where it fits in the broader picture. Our main result is that this category is the ``splitting category'' of a certain Azumaya algebra in the ordinary interpolation category $\uRep(\Sp_t(\bF))$ of finite symplectic groups. This construction of the category has a few advantages over previous ones, and also (in a sense) explains this category: it exists because the Brauer group of $\uRep(\Sp_t(\bF))$ is non-trivial.

\subsection{Splitting categories}

Fix a field $k$, which we assume here to be algebraically closed for simplicity. Let $\fC$ be a symmetric $k$-linear tensor category (Definition~\ref{defn:tencat}) and let $A$ be an Azumaya algebra in $\fC$ (\S \ref{ss:azumaya}). We would like to construct an extension of $\fC$ in which $A$ is split, i.e., of the form $\uEnd(X)$ for some rigid object $X$ of $\fC$; moreover, we would like this extension to be uniquely determined by some reasonably property.

There are two issues. First, there is an invariant $\eta(A) \in \Pic(\fC)[2]$ that must vanish for such an extension to exist. So supose this is the case. The second issue is that we cannot uniquely pin down the object $X$, in general. For example, in the category of super vector spaces, we have $\uEnd(X)=\uEnd(X[1])$, where $[1]$ is the parity shift functor.

We introduce a notion that neatly deals with both of the above issues. An \defn{S-algebra}\footnote{For ``splittable algebra.''} in $\fC$ is an algebra $A$ equipped with a linear map $\epsilon \colon A \to \bone$ and an invariant element $\pi \in A \otimes A$ satisfying some conditions (see Definition~\ref{defn:Salg}). The primary example is $A=\uEnd(X)$, where $\epsilon$ is the trace map and $\pi$ is the symmetry of $X \otimes X$. We show that an algebra $A$ admits an S-structure if and only if it is an Azumaya algebra with $\eta(A)=1$. Moreover, the choice of S-structure on $A$ precisely eliminates the ambiguity of $X$ mentioned above. S-algebras have one other advantage over general Azumaya algebras, in that their definition is more elementary and purely diagrammatic.

Let $A$ be an S-algebra in $\fC$. A \defn{splitting category} for $A$ is a symmetric tensor functor $\Phi \colon \fC \to \fD$ together with an isomorphism of S-algberas $\Phi(A) \to \uEnd(X)$ for some rigid object $X$ of $\fD$. We say a splitting category is \defn{universal} if it is initial in the 2-category of splitting categories. The following is our main theorem on these objects:

\begin{theorem} \label{mainthm}
A separable S-algebra admits a universal splitting category.
\end{theorem}

See \S \ref{ss:separable} for the definition of separable; we note that if $\udim(A) \ne 0$ or $\fC$ is semi-simple then separability is automatic. In fact, we prove more precise results. For instance, the universal splitting category $\fD$ is $\bZ$-graded, and its degree~0 piece is $\fC$ itself, and we show that many properties of $\fC$ (such as abelian) are inherited by $\fD$. We also discuss a variant (\S \ref{ss:msplit}) that gives a $\bZ/m$-graded category instead.

We mention one particularly interesting example (see in \S \ref{ss:pgl} for details). Let $\fD$ be the Deligne--Milne interpolation category $\uRep(\GL_t)$ (with $t \ne 0$), and let $\fC=\uRep(\PGL_t)$ be its degree~0 piece. The endomorphism algebra $A$ of the standard representation of $\GL_t$ is an S-algebra in $\fC$, and $\fD$ is its universal splitting category. Now suppose that $A'$ is an S-algebra in an arbitrary tensor category $\fC'$. Then there is a unique tensor functor $\fC \to \fC'$ mapping $A$ to $A'$; this is the universal property of $\fC$. The universal splitting category $\fD'$ of $A'$ is then the push-out of $\fC'$ along $\fC \to \fD$. Thus Theorem~\ref{mainthm} is equivalent to the statement that such push-outs exist in the 2-category of tensor categories.

\subsection{Kriz's category}

Assume now that $k$ has characteristic~0. Suppose $V$ is a finite dimensional symplectic vector space over a finite field $\bF$ of odd characteristic, and let $\psi$ be a non-trivial additive character of $\bF$ valued in $k^{\times}$. We can then create the twisted group algebra of $V$. The underlying vector space is the usual group algebra $k[V]$, but the multiplication is given by
\begin{displaymath}
[v] \cdot [w] = \psi(\langle v, w \rangle) \cdot [v+w].
\end{displaymath}
This algebra carries a natural S-structure (\S \ref{ss:SpS}). In fact, it is split, as it is naturally isomorphic to the endomorphism algebra of the oscillator representation of the finite symplectic.

Now, consider the interpolation category $\fC=\uRep(\Sp_t(\bF))$ for the finite symplectic groups. There are several ways of constructing this category; we discuss the ultraproduct approach in \S \ref{s:ultra}, and the oligomorphic approach in \S \ref{s:oligo}. The S-algebras from the previous paragraph interpolate to an S-algebra $A$ in $\fC$. However, $A$ is not split in the category $\fC$ (at least if $t$ is generic). The following is our main result:

\begin{theorem} \label{mainthm2}
(One version of) Kriz's interpolation category is equivalent to the universal splitting category $\fD$ of $A$.
\end{theorem}

We only prove this theorem for generic $t$, but it certainly holds more generally. The proof rests on a result of Deligne \cite{DeligneLetter1} (Theorem~\ref{thm:mod4}) about the representation theory of finite symplectic groups.

Our construction of the category $\fD$ has some advantages over Kriz's original approach (via her theory of T-algebras):
\begin{enumerate}
\item Our construction of $\fD$ is in a sense relative to the construction of $\fC$: one can use whatever method one prefers (ultraproducts, oligomorphic groups, T-algebras) to build $\fC$, and then $\fD$ pops out with little work.
\item Kriz only dealt with non-singular parameter values of $t$, where $\fC$ is semi-simple, but our approach only requires $t \ne 0$ (and even gives quite a bit at $t=0$). For instance, for any $t \ne 0$, our construction produces a pre-Tannakian category $\fD$.
\item Since our $\fD$ is produced using a universal construction, it comes with a mapping property. This more or less says that giving a tensor functor $\Phi$ out of $\fD$ is equivalent to giving one $\Phi_0$ out of $\fC$ together with a splitting of $\Phi_0(A)$. An explicit mapping property for $\fC$ does not yet exist in the literature, but \cite{EAH} gives a very detailed mapping property in the similar case of $\uRep(\GL_t(\bF))$.
\item Our construction does not require one to ever write down the oscillator representation; one only needs the twisted group algebras discussed above, which are far easier to deal with.
\end{enumerate}
There is one other result we obtain from our perspective: we show that the Brauer group of $\fC$ contains a subgroup of order four. Again, we only prove this for generic $t$, and the proof rests on Deligne's work mentioned above.

\subsection{Related work}

The study of $G$-graded extensions of (not necessarily symmetric) tensor categories has received a fair amount of attention in the literature; see, e.g., \cite{DN, ENO, GMPPS, JFR}. The approach of these papers is homotopy theoretic: $G$-graded extensions of $\fC$ correspond to maps from the classifying space $BG$ to a higher groupoid associated to $\fC$.

We do not expect our results on splitting categories to be surprising to experts in this area. However, we nonetheless feel that there could be value in our approach; we give three reasons. First, our construction of the universal splitting category $\fD$ is elementary and canonical: the associators are ``trivial'' (induced from the given category $\fC$), and symmetry isomorphisms, while not trivial, are given by simple explicit formulas. Second, our results are perhaps easier to apply in concrete situations: the obstructions are explicit conditions in the given category $\fC$, as opposed to cohomology classes in an associated space. And third, we characterize our splitting category $\fD$ by a universal property, which seems to be new.

\subsection{Outline}

In \S \ref{s:tencat}, we review material on tensor categories. In \S \ref{s:azumaya}, we study S-algebras, Azumaya algebras, and how they relate to one another. In \S \ref{s:split}, we prove our main theorem on splitting categories (Theorem~\ref{mainthm}), and in \S \ref{s:split2} we provide some supplementary material to it. In \S \ref{s:ultra}, we prove Theorem~\ref{mainthm2} using the ultraproduct approach, and in \S \ref{s:oligo} we make some comments on the oligomorphic approach.

\subsection{Notation}

We list some important notation:
\begin{description}[align=right,labelwidth=2cm,leftmargin=!]
\item [ $k$ ] the base field
\item [ $\fC$ ] a tensor category (always symmetric)
\item [ $\bone$ ] the unit object of a tensor category
\item [ $\tau$ ] the symmetry of a tensor category
\item [ $\Gamma$ ] the invariants functor, $\Hom_{\fC}(\bone, -)$
\item [ $\udim$ ] categorical dimension
\item [ $V^*$ ] the dual of a rigid object $V$
\item [ $A^{\op}$ ] the opposite algebra of $A$
\item [ $\Mod_A$ ] the category of left $A$-modules
\end{description}

\subsection*{Acknowledgments}

We thank Pierre Deligne, Pavel Etingof, Nate Harman, Sophie Kriz, and Noah Snyder for helpful conversations.

\section{Tensor categories} \label{s:tencat}

\subsection{General definitions}

Fix a field $k$ for the duration of the paper.

\begin{definition} \label{defn:rigid}
Let $\fC$ be a symmetric monoidal category, and let $X$ be an object of $\fC$. A \defn{dual} of $X$ is an object $Y$ equipped with maps
\begin{displaymath}
\alpha \colon \bone \to X \otimes Y, \qquad \beta \colon Y \otimes X \to \bone,
\end{displaymath}
called \defn{co-evaluation} and \defn{evaluation} satisfying the two \defn{zig-zag} identites. Here $\bone$ denotes the monoidal unit. The first zig-zag identity states that the composition
\begin{displaymath}
\xymatrix@C=3em{
\bone \otimes X \ar[r]^-{\alpha \otimes \id} &
X \otimes Y \otimes X \ar[r]^-{\id \otimes \beta} &
X \otimes \bone }
\end{displaymath}
is the identity (after identifying the outside objects with $X$), and the second identity is similar. We say that $X$ is \defn{rigid} if it has a dual; in this case, the dual is unique (up to isomorphism), and we denote it $X^*$. If $X$ has a dual, we say that $X$ is \defn{rigid}. If every object of $\fC$ has a dual, we say that $\fC$ is \defn{rigid}.
\end{definition}

An object in a symmetric monoidal category $\fC$ is \defn{invertible} if it is rigid and the (co-)evaluation maps are isomorphisms. We let $\Pic(\fC)$ be the abelian group of isomorphism classes of invertible objects.

\begin{definition} \label{defn:tencat}
A \defn{tensor category} is an additive $k$-linear category $\fC$ equipped with a $k$-bilinear symmetric monoidal structure such that $\End(\bone)=k$. A tensor category $\fC$ is \defn{pre-Tannakian} if it is rigid, $\Hom$-finite (i.e., all $\Hom$ spaces are finite dimensional), abelian, and all objects have finite length. A \defn{tensor functor} is a $k$-linear symmetric monoidal functor.
\end{definition}

Suppose $X$ is a rigid object in a tensor category. Let $\alpha$ and $\beta$ be the (co-)evaluation maps. Pre-composing $\beta$ with the symmetry, we regard its domain as $X \otimes Y$. The composition $\beta \circ \alpha$ is then an endomorphism of $\bone$, and thus an element of $k$. The \defn{categorical dimension} of $X$, denoted $\udim(X)$, is this scalar.

If $L$ is invertible object in a tensor category then $\End(L)=\End(\bone)=k$. In particular, the symmetry of $L^{\otimes 2}$ is multiplication by $\pm 1$; we refer to this as the \defn{sign} of $L$.

We refer to \cite{EGNO} for general background on tensor categories. We warn the reader that our ``pre-Tannakian category'' corresponds with their ``symmetric tensor category.''

\subsection{Algebras}

Let $\fC$ be a tensor category. An \defn{algebra} in $\fC$ will always mean an associative and unital algebra; typically, our algebras will not be commutative. For an arbitrary object $X$ of $\fC$, we let
\begin{displaymath}
\Gamma(X) = \Hom_{\fC}(\bone, X),
\end{displaymath}
which we call the \defn{invariant space} of $X$. If $A$ is an algebra then so is $\Gamma(A)$. Moreover, if $a \in \Gamma(A)$ then there are induced left and right multiplication morphisms
\begin{displaymath}
\lambda_a, \rho_a \colon A \to A
\end{displaymath}
in $\fC$. These will play an important role in our discussion.

Let $A$ be an algebra. We write $\Mod_A$ for the category of left $A$-modules. If $a \in \Gamma(A)$ and $M$ is an $A$-module then there is a left multiplication morphim $\lambda_a \colon M \to M$. We say that an $A$-module  $M$ is \defn{relatively free} if it has the form $A \otimes X$ for some $X \in \fC$, and \defn{relatively projective} if it is a summand of a relatively free module. We note that if $\fC$ is Karoubian then so is the category of relatively projective $A$-modules.

\subsection{Inverting an object}

Let $\fC$ be a symmetric monoidal category and let $L$ be an object of $\fC$. Define a new monoidal category $\fC[1/L]$, as follows. The objects of $\fC[1/L]$ are pairs $(X, n)$ where $X$ is an object of $\fC$ and $n$ is an integer. Morphisms are defined as follows
\begin{displaymath}
\Hom_{\fC[1/L]}((X,n), (Y,m)) = \varinjlim_r \Hom_{\fC}(L^{\otimes (r+n)} \otimes X, L^{\otimes (r+m)} \otimes Y).
\end{displaymath}
The monoidal operation on $\fC[1/L]$ is defined on objects by
\begin{displaymath}
(X,n) \otimes (Y,m) = (X \otimes Y, n+m),
\end{displaymath}
and in the obvious manner on morphisms. There is a natural monoidal functor $\Phi \colon \fC \to \fC[1/L]$ defined on objects by $\Phi(X)=(X,0)$. We require the following result:

\begin{proposition} \label{prop:invert1}
Suppose the 3-cycle $(1\;2\;3)$ acts trivially on $L^{\otimes 3}$. Then $\fC[1/L]$ is naturally a symmetric monoidal category, $\Phi$ is naturally a symmetric monoidal functor, and $\Phi(L)$ is invertible.
\end{proposition}

\begin{proof}
This is well-known; see \cite[Theorem~4.3]{Voevodsky} or \cite[\S 4.2.2]{Robalo}.
\end{proof}

\begin{remark}
If $M$ is an invertible object in a symmetric monoidal category then $(1\;2\;3)$ acts trivially on $M^{\otimes 3}$. Thus this condition on $L$ is necessary if we want to invert $L$.
\end{remark}

\subsection{Trivializing an invertible object} \label{ss:triv1}

Let $\fC$ be a tensor category and let $L$ be an invertible object of $\fC$ such that the symmetry of $L^{\otimes 2}$ is trivial. We now define a new symmetric monoidal category $\fC\{L\}$ in which $L$ becomes isomorphic to the monoidal unit. This is sometimes called ``de-equivariantization.''

The objects of $\fC\{L\}$ are the same as those of $\fC$; for notational clarity, if $X$ is an object of $\fC$, we write $[X]$ for the corresponding object of $\fC\{L\}$. Morphisms are defined by
\begin{displaymath}
\Hom_{\fC\{L\}}([X], [Y]) = \bigoplus_{n \in \bZ} \Hom_{\fC}(X, L^{\otimes n} \otimes Y).
\end{displaymath}
Composition is defined in the obvious manner. The tensor product on $\fC\{L\}$ is defined on objects by $[X] \otimes [Y]=[X \otimes Y]$. Suppose that
\begin{displaymath}
[f_1] \colon [X_1] \to [Y_1], \qquad [f_2] \colon [X_2] \to [Y_2]
\end{displaymath}
are morphisms in $\fC\{L\}$, represented by morphisms
\begin{displaymath}
f_1 \colon X_1 \to L^{\otimes n} \otimes Y_1, \qquad f_2 \colon X_2 \to L^{\otimes m} \otimes Y_2
\end{displaymath}
in $\fC$. The tensor product is the morphism in $\fC\{L\}$ represented by the composition
\begin{displaymath}
\xymatrix@C=3em{
X_1 \otimes X_2 \ar[r]^-{f_1 \otimes f_2} \ar[r] &
L^{\otimes n} \otimes Y_1 \otimes L^{\otimes m} \otimes Y_2 \ar@{=}[r] &
L^{\otimes n+m} \otimes Y_1 \otimes Y_2, }
\end{displaymath}
where the second isomorphism is the symmetry isomorphism in $\fC$. This definition is extended linearly to general morphisms.

Finally, the symmetry on $\fC\{L\}$ is simply induced from the one on $\fC$. That is, the isomorphism $[X_1] \otimes [X_2] \to [X_2] \otimes [X_1]$ is simply the corresponding isomorphism from $\fC$. This clearly satisfies the hexagon axioms. The one subtlety here is functoriality. Consider morphisms $[f_1]$ and $[f_2]$ as above. We must show that the square
\begin{displaymath}
\xymatrix@C=3em{
[X_1] \otimes [X_2] \ar[r]^{[f_1] \otimes [f_2]} \ar@{=}[d] &
[Y_1] \otimes [Y_2] \ar@{=}[d] \\
[X_2] \otimes [X_1] \ar[r]^{[f_2] \otimes [f_1]} &
[Y_2] \otimes [Y_1] }
\end{displaymath}
commutes. This diagram comes from the following one in $\fC$
\begin{displaymath}
\xymatrix@C=3em{
X_1 \otimes X_2 \ar[r]^-{f_1 \otimes f_2} \ar@{=}[d] &
L^{\otimes n} \otimes Y_1 \otimes L^{\otimes m} \otimes Y_2 \ar@{=}[d] \ar@{=}[r] &
L^{\otimes (n+m)} \otimes Y_1 \otimes Y_2 \ar@{=}[d] \\
X_2 \otimes X_1 \ar[r]^-{f_2 \otimes f_1} &
L^{\otimes m} \otimes Y_2 \otimes L^{\otimes n} \otimes Y_1 \ar@{=}[r] &
L^{\otimes (m+n)} \otimes Y_2 \otimes Y_1 }
\end{displaymath}
This diagram commutes in $\fC$, provided the rightmost isomorphism is $\tau \otimes \tau'$, where $\tau$ is the symmetry on $L^{\otimes (n+m)}=L^{\otimes n} \otimes L^{\otimes m}$ and $\tau'$ is the symmetry on $Y_1 \otimes Y_2$. However, since the symmetry on $L^{\otimes 2}$ is trivial, it follows that $\fS_n$ acts trivially on $L^{\otimes n}$ for all $n$, and so $\tau$ is the identity morphism of $L^{\otimes (n+m)}$. This implies that the previous diagram in $\fC\{L\}$ commutes.

We have thus constructed a symmetric monoidal structure on $\fC\{L\}$. There is a natural symmetric monoidal functor $\Phi \colon \fC \to \fC\{L\}$ defined by $\Phi(X)=[X]$. The key property of this construction is that we have a natural isomorphism $i \colon \Phi(L) \to \bone$; this is the map $[L] \to [\bone]$ represented by the identity map $L \to L \otimes \bone$. One can show that this construction is universal with respect to this property.

\begin{remark}
Let $R=\bigoplus_{n \in \bZ} L^{\otimes n}$, regarded as an ind-object in $\fC$. This is a commutative algebra in $\Ind(\fC)$; it is essentially a ring of Laurent series (note that for $n \ge 0$ we have $L^{\otimes n}=\Sym^n(L)$ since the symmetry of $L^{\otimes 2}$ is trivial). Our category $\fC\{L\}$ is equivalent to the full subcategory of $\Mod_R$ spanned by objects of the form $R \otimes X$ with $X \in \fC$. Moreover, the tensor product on $\fC\{L\}$ simply corresponds to $\otimes_R$ from this point of view.
\end{remark}

\begin{remark}
We warn the reader that if $L=\bone$ is already trivial then $\fC\{L\}$ is not $\fC$, as the morphism spaces become larger; for example, $\End_{\fC\{L\}}(\bone)=k[t,t^{-1}]$.
\end{remark}

\subsection{Trivializing more general objects} \label{ss:triv2}

Let $\fC$ be a tensor category and let $L$ be an object (not necessarily invertible) for which the symmetry of $L^{\otimes 2}$ is trivial. As noted above, this implies that the symmetric group $\fS_3$ acts trivially on $L^{\otimes 3}$. We can therefore apply the first construction above to obtain a symmetric monoidal functor $\Phi_1 \colon \fC \to \fC[1/L]$ such that $L'=\Phi_1(L)$ is invertible. Since $\Phi_1$ is symmetric monoidal, it follows that the symmetry of $(L')^{\otimes 2}$ is trivial, and so we can apply the second construction to obtain a symmetric monoidal functor $\Phi_2 \colon \fC[1/L] \to (\fC[1/L])\{L'\}$ equipped with an isomorphism $\Phi_2(L') \to \bone$. We put $\fC\{L\}=(\fC[1/L])\{L'\}$ for notational simplicity. We have a symmetric monoidal functor $\Phi \colon \fC \to \fC\{L\}$, and a natural isomorphism $\Phi(L) \to \bone$. Again, one can show that this construction is universal.

\section{S-algebras and Azumaya algebras} \label{s:azumaya}

\subsection{S-algebras}

Fix a tensor category $\fC$ for the duration of \S \ref{s:azumaya}. The following is a central concept in this paper:

\begin{definition} \label{defn:Salg}
An \defn{S-structure} on an algebra $A$ in $\fC$ is a pair $(\pi, \epsilon)$ consisting of an invariant element $\pi \in \Gamma(A \otimes A)$ and a map $\epsilon \colon A \to \bone$ in $\fC$ such that the following identities hold:
\begin{displaymath}
\tau=\lambda_{\pi} \rho_{\pi}, \qquad \mu=(\id \otimes \epsilon) \rho_{\pi}.
\end{displaymath}
An \defn{S-algebra} is an algebra equipped with an S-structure.
\end{definition}

We recall that $\lambda$ and $\rho$ are left and right multiplication operators, $\mu$ is the multiplication map on $A$, and $\tau$ is the symmetry of $A^{\otimes 2}$. We note that if $(\pi, \epsilon)$ is an S-structure on $A$ then so is $(-\pi, -\epsilon)$; in Proposition~\ref{prop:Sclass}, we will see that it is the only other one (under mild hypotheses). The primary examples of S-algebras are endomorphism algebras:

\begin{example}
Let $V$ be a rigid object of $\fC$. Recall that $\uEnd(V)=V \otimes V^*$ is the internal endomorphism algebra of $V$. Let $\epsilon \colon \uEnd(V) \to \bone$ be the trace map, i.e., the evaluation map for $V$. We have a natural identification
\begin{displaymath}
\Gamma(\uEnd(V) \otimes \uEnd(V)) = \End(V \otimes V),
\end{displaymath}
and we let $\pi$ be the element here corresponding to the symmetry automorphism of $V \otimes V$. One readily verifies that $(\pi, \epsilon)$ is an S-structure, and in this way we regard $\uEnd(V)$ as an S-algebra.
\end{example}

We now establish a few basic results about S-algebras.

\begin{proposition}
Let $(\pi, \epsilon)$ be an S-structure on $A$. Then $\pi^2=1$ and $\tau(\pi)=\pi$.
\end{proposition}

\begin{proof}
Evaluating the identity $\tau=\lambda_{\pi} \rho_{\pi}$ at the identity element 1 of $A \otimes A$, and using $\tau(1)=1$, we see that $\pi^2=1$. Evaluating at $\pi$, we see that $\tau(\pi)=\pi^3=\pi$.
\end{proof}

Note that since $\pi^2=1$, the operator $\lambda_{\pi} \rho_{\pi}$ is conjugation by $\pi$. Thus in an S-algebra, the symmetry $\tau$ of $A \otimes A$ is inner.

\begin{proposition}
Let $A$ be an S-algebra. Then the maps $\pi \colon \bone \to A \otimes A$ and $\epsilon \mu \colon A \otimes A \to \bone$ define a symmetric self-duality of $A$. In particular, $A$ is a rigid object of $\fC$.
\end{proposition}

\begin{proof}
We must show that the composition
\begin{displaymath}
\xymatrix@C=3em{
A \ar[r]^-{\id \otimes \pi} & A \otimes A \otimes A \ar[r]^-{\epsilon \mu \otimes \id} & A }
\end{displaymath}
is the identity. (This is the first zig-zag identity; the proof of the second is similar.) Call the above map $\phi$. We have
\begin{displaymath}
\phi(x) = (\epsilon \otimes \epsilon \otimes \id)((x \otimes 1 \otimes 1)\pi_2 \pi_1),
\end{displaymath}
where $\pi_1=\pi \otimes 1$ and $\pi_2 = 1 \otimes \pi$. Factor $\epsilon \otimes \epsilon \otimes \id$ as $(\epsilon \otimes \id)(\id \otimes \epsilon \otimes \id)$, and note that $(x \otimes 1 \otimes 1)$ pulls out of the latter map. We thus obtain
\begin{displaymath}
\phi(x) = (\epsilon \otimes \id)\big( (x \otimes 1) \cdot (\id \otimes \epsilon \otimes \id)(\pi_2 \pi_1) \big).
\end{displaymath}
We have
\begin{displaymath}
(\id \otimes \epsilon \otimes \id)(\pi_2 \pi_1) = (\mu \otimes \id)(\pi_2) = \pi.
\end{displaymath}
Thus
\begin{displaymath}
\phi(x) = (\epsilon \otimes \id)((x \otimes 1) \pi) = (\id \otimes \epsilon)((1\otimes x)\pi)=\mu(1 \otimes x)=x.
\end{displaymath}
We have thus shown that $\phi$ is the identity, and so we have a duality. The co-evaluation map $\pi \colon \bone \to A \otimes A$ is symmetric since $\tau(\pi)=\pi$. Similarly, the evaluation map $\epsilon \mu=(\epsilon \otimes \epsilon) \rho_{\pi}$ is symmetric.
\end{proof}

We define the \defn{categorical degree} of an S-algebra, denoted $\udeg(A)$, to be the quantity $\epsilon \circ \eta \in \End(\bone)=k$, where $\eta \colon \bone \to A$ is the unit. We often abbreviate this to $\epsilon(1)$. We note that replacing $(\pi, \epsilon)$ with $(-\pi, -\epsilon)$ negates the categorical degree. We also note that the categorical degree of $\uEnd(V)$ coincides with the categorical dimension of $V$. The following corollary is an analog of the relation between dimension and degree for central simple algebras.

\begin{corollary} \label{cor:dim-deg}
For an S-algebra $A$ we have $\udim(A) = \udeg(A)^2$.
\end{corollary}

\begin{proof}
The categorical dimension of $A$ is the composition of the evaluation and co-evaluation maps. We thus have
\begin{displaymath}
\udim(A) = (\epsilon \mu)(\pi) = (\epsilon \otimes \epsilon)(\pi^2)=\epsilon(1)^2,
\end{displaymath}
since $\pi^2=1=1 \otimes 1$.
\end{proof}

Let $A$ be an S-algebra. The symmetric group $\fS_n$ obviously acts on $A^{\otimes n}$. We now show that there is actually an action of the larger group $\fS_n \times \fS_n$; this is a key feature of S-algebras. To begin, define $\pi_1=\pi \otimes 1$ and $\pi_2=1 \otimes \pi$ in $\Gamma(A^{\otimes 3})$.

\begin{proposition} \label{prop:braid}
The braid equation $(\pi_1 \pi_2)^3=1$ holds.
\end{proposition}

\begin{proof}
Evaluating the identity $\tau=\lambda_{\pi} \rho_{\pi}$ at~1 gives $\pi^2=1$. Note that the equation $\tau=\lambda_{\pi} \rho_{\pi}$ means that $\tau$ is conjugation by $\pi$. We thus have $\tau(\pi)=\pi$. Next,
\begin{displaymath}
\pi_1 \pi_2 \pi_1 = \pi_2 \cdot (2\;3) \pi_1 \cdot \pi_1 = \pi_2 \pi_1 \cdot (1\;2)(2\;3)\pi_1.
\end{displaymath}
Here we are using that conjugation by $\pi_1$ and $\pi_2$ give the actions of the transpositions $(1\;2)$ in $(2\;3)$ on $A^{\otimes 3}$. Since $(1\;2)(2\;3)\pi_1=\pi_2$, we obtain the braid equation.
\end{proof}

Fix $n \ge 2$. For $1 \le i \le n-1$, let $\pi_i \in \Gamma(A^{\otimes n})$ be the element $\pi_i=1^{\otimes (i-1)} \otimes \pi \otimes 1^{\otimes (n-i-1)}$. It follows from the above proposition that these elements also satisfy the braid relation, i.e., $(\pi_i \pi_{i+1})^3=1$ for $1 \le i \le n-2$. We therefore have a group homomorphism
\begin{displaymath}
\pi \colon \fS_n \to \Gamma(A^{\otimes n})^{\times}, \qquad s_i \mapsto \pi_i,
\end{displaymath}
where $\fS_n$ is the symmetric group, the target is the unit group of the algebra $\Gamma(A^{\otimes n})$, and the $s_i$'s are the Coxeter generators of $\fS_n$. We write $\pi(\sigma)$ for the image of $\sigma \in \fS_n$, so $\pi(s_i)=\pi_i$ by definition. This homomorphism allows us to define an action of $\fS_n \times \fS_n$ on $A^{\otimes n}$ by letting $(\sigma, \sigma')$ act by $\lambda_{\pi(\sigma)} \rho_{\pi(\sigma')}^{-1}$; the diagonal $\fS_n$ acts in the usual manner by permuting tensor factors.

The final topic we discuss here is the notion of splitting of S-algebras. Let $A$ be an arbitrary S-algebra in $\fC$. An \defn{S-splitting} of $A$ is an isomorphism of S-algebras $A \cong \uEnd(V)$ for some rigid object $V$; we note then that $\udim(V)=\udeg(A)$. We say that $A$ is \defn{S-split} if such an isomorphism exists. We use the S- prefix here for clarity, as we say that an algebra $A$ is \defn{split} if it is isomorphic to $\uEnd(V)$ for some rigid $V$. We give two simple examples.

\begin{example}
(a) Let $\fC$ be the category of complex vector spaces, and let $A=\bone$. This admits two S-structures $(\pm 1, \pm 1)$. The isomorphism $A \cong \uEnd(\bC)$ is an S-splitting in the $\pi=1$ case. The $\pi=-1$ case is not S-split, as no vector space has dimension $-1$.

(b) Now let $\fD$ be the category of super complex vector spaces. We can regard the algebra $A$ from (a) in $\fD$ as well. The isomorphisms
\begin{displaymath}
A \cong \uEnd(\bC^{1|0}), \qquad A \cong \uEnd(\bC^{0|1})
\end{displaymath}
are S-splittings when $\pi=1$ and $\pi=-1$ respectively. We thus see that while the $\pi=-1$ case is not S-split in $\fC$, it becomes so in the larger category $\fD$; we explore this idea systematically in \S \ref{s:split}.
\end{example}

\subsection{Separable algebras} \label{ss:separable}

Recall that an algebra $A$ is \defn{separable} if the multiplication map $\mu \colon A \otimes A \to A$ has a splitting in the category of $(A,A)$-bimodules; here the bimodule structure on the source is $a(x \otimes y)b=ax \otimes yb$. One important property of separable algebras is that if $M$ is a right $A$-module and $N$ is a left $A$-module then $M \otimes_A N$ is a summand of $M \otimes N$; in particular, the tensor product exists if $\fC$ is Karoubian. We will require the following (well-known) result:

\begin{proposition} \label{prop:sepss}
Let $A$ be a separable algebra. Then every $A$-module is a summand of one of the form $A \otimes X$, with $X \in \fC$. In particular, if $\fC$ is semi-simple then $\Mod_A$ is too.
\end{proposition}

\begin{proof}
Let $M$ be a left $A$-module. Then $M=A \otimes_A M$ is a summand of $(A \otimes A) \otimes_A M=A \otimes M$, as required. (Note that the module structure on $A \otimes M$ here just comes from $A$.) Suppose $\fC$ is semi-simple. Then $\Hom_A(A \otimes X, -)=\Hom_{\fC}(X, -)$ is exact, and so $A \otimes X$ is a projective $A$-module. Hence every $A$-module is projective. Finite length is inherited from $\fC$.
\end{proof}

We have a useful criterion for separability of S-algebras:

\begin{proposition} \label{prop:separable}
Let $A$ be an S-algebra.
\begin{enumerate}
\item If there exists an element $u \in \Gamma(A)$ with $\epsilon(u)=1$ then $A$ is separable.
\item Such an element $u$ exists if either $\udim(A) \ne 0$, or $\fC$ is semi-simple.
\end{enumerate}
\end{proposition}

\begin{proof}
(a) Define $\sigma \colon A \to A \otimes A$ by $\sigma(x) = (x \otimes u) \pi$. We have
\begin{displaymath}
\sigma(axb) = (a \otimes 1)(x \otimes u)(b \otimes 1) \pi = (a \otimes 1)(x \otimes u)\pi(1 \otimes b) = (a \otimes 1) \sigma(x) (1 \otimes b)
\end{displaymath}
which shows that $\sigma$ is a map of bimodules. We have
\begin{displaymath}
\mu(\sigma(x)) = (\id \otimes \epsilon)((x \otimes u)\pi^2) = x.
\end{displaymath}
Thus $\sigma$ is a section of $\mu$, and so $A$ is separable.

(b) If $\udim(A) \ne 0$ then $\epsilon(1) \ne 0$, and we can take $u=\epsilon(1)^{-1}$. If $\fC$ is semi-simple  and $A$ is non-zero then the surjection $\epsilon \colon A \to \bone$ splits, which yields the element $u$. (Note that $\epsilon$ is non-zero since $\epsilon \mu$ is a perfect pairing. Since $\bone$ is simple, $\epsilon$ is surjective.)
\end{proof}

\begin{corollary} \label{cor:azumss}
If $\fC$ is semi-simple and $A$ is an S-algebra then $\Mod_A$ is semi-simple.
\end{corollary}

\subsection{Azumaya algebras} \label{ss:azumaya}

An \defn{Azumaya algebra} in $\fC$ is a rigid algebra $A$ for which the natural map
\begin{displaymath}
A \otimes A^{\rm op} \to \uEnd(A), \qquad a \otimes b \mapsto (x \mapsto axb)
\end{displaymath}
is an isomorphism. If $A$ is an Azumaya algebra then so is the opposite algebra $A^{\op}$. If $B$ is a second Azumaya algebra then $A \otimes B$ is also an Azumaya algebra. If $V$ is a rigid object of $\fC$ then $A=\uEnd(V)$ is an Azumaya algebra; we say that algebras of this form are \defn{split}. See \cite{Pareigis4} for these statements.

Let $A$ be an Azumaya algebra, and let ${}_A\Mod_A$ be the category of $(A,A)$-bimodules. In what follows, we assume that $\otimes_A$ exists and is well-behaved on this category; this is the case if, e.g., $\fC$ is abelian, or $\fC$ is Karoubian and $A$ is separable. The functor
\begin{displaymath}
\fC \to {}_A\Mod_A, \qquad X \mapsto X \otimes A,
\end{displaymath}
is an equivalence; in fact, this property characterizes Azumaya algebras \cite[Proposition~1]{Pareigis4}. This equivalence is in fact one of $\fC$-module categories. It follows that the category of $\fC$-linear endofunctors of ${}_A\Mod_A$ is equivalent to $\fC$; moreover, this equivalence is monoidal, using composition of endofunctors and the tensor product on $\fC$.

Now, let $B=A \otimes A$. This algebra carries a natural symmetry isomorphism $\tau \colon B \to B$. Given a $B$-module $M$, define $M^{\dag}$ to be the the $B$-bimodule with underlying object $M$, but where the right action of $B$ is twisted by the symmetry $\tau$. Then $M \mapsto M^{\dag}$ is a $\fC$-linear autoequivalence of ${}_B\Mod_B$. By the above remarks, we have a functorial isomorphism $M^{\dag} \cong L \otimes M$ for some object $L$ of $\fC$, which is unique up to isomorphism. Moreover, since the square of  $(-)^{\dag}$ is canonically the identity functor, there is a natural isomorphism $i \colon L^{\otimes 2} \to \bone$. We let $\eta(A)$ be the class of $L$ in $\Pic(\fC)[2]$. This is an important invariant of $A$.

In fact, there is a more refined invariant we can attach to $A$; we thank Deligne for this remark. Define $\wt{\Pic}(\fC)[2]$ to be the group of isomorphism classes of pairs $(M, j)$, where $M$ is an invertible object of $\fC$ and $j \colon M^{\otimes 2} \to \bone$ is an isomorphism\footnote{We note that $\wt{\Pic}(-)[2]$ is one symbol; it is not the 2-torsion in a group $\wt{\Pic}(-)$.}. We have a natural short exact sequence
\begin{displaymath}
1 \to k^{\times}/(k{^\times})^2 \to \wt{\Pic}(\fC)[2] \to \Pic(\fC)[2] \to 1
\end{displaymath}
We define $\eta(A)$ to be the class of $(L,i)$ in $\wt{\Pic}(\fC)[2]$. We note that if $k$ is algebraically closed (or simply quadratically closed) then $\eta(A)$ and $\tilde{\eta}(A)$ carry the same information, but in general $\tilde{\eta}(A)$ carries more information.

The condition $\eta(A)=1$ means that there is an isomorphism $i \colon B \to B^{\dag}$ of $B$-bimodules; the condition $\tilde{\eta}(A)=1$ means that there is such an isomorphism satisfying $i^2=\id$. Thus both conditions make sense even in settings where ${}_A\Mod_A$ lacks a tensor structure.

\begin{example}
We give a few examples of the above invariants.
\begin{enumerate}
\item If $A=\uEnd(V)$ is split then $\tilde{\eta}(A)=1$. Indeed, we have
\begin{displaymath}
B = V \otimes V^* \otimes V \otimes V^*.
\end{displaymath}
Let $i \colon B \to B^{\dag}$ be the map switching the two $V^*$ factors and fixing the two $V$ factors. Then $i$ is an isomorphism of $B$-bimodules and $i^2=\id$.
\item Let $\fC$ be the category of complex super vector spaces. Let $A=\bC[x]/(x^2=1)$ where $x$ is odd, which is an Azumaya algebra in $\fC$. Then we have an isomorphism of $B$-modules $B \cong \bC^{0|1} \otimes B^{\dag}$, and so $\eta(A)$ is the non-trivial class in $\Pic(\fC) \cong \bZ/2$. Note that $\eta(A)$ has sign $-1$.
\item Let $\fC$ be the category of real super vector spaces. Let $A=\bR[x]/(x^2=1)$ be the real form of the above algebra. Then $\eta(A^{\otimes 2})$ is trivial, but $\tilde{\eta}(A^{\otimes 2})$ is the non-trivial square class in $\bR^{\times}$. We thank Deligne for this example.
\item Let $\fC$ be the category of complex $\bZ/2 \times \bZ/2$ graded vector spaces, where the symmetry uses the super sign rule with respect to total degree. Let $A=\bC[x,y]/(x^2=y^2=1)$, where $x$ has degree $(1,0)$ and $y$ has degree $(0,1)$. Then $\eta(A)$ is the class of the one-dimensional vector space placed in bi-degree $(1,1)$. In particular, $\eta(A)$ is non-trivial but has sign $+1$. We thank S.~Kriz for this example.
\qedhere
\end{enumerate}
\end{example}

\begin{remark}
Suppose that $\fC$ is pre-Tannkian and $k$ is algebraically closed. Let $A$ be an Azumaya algebra in $\fC$ such that $\udim(A) \ne 0$ and $\eta(A)$ has sign $+1$ (which follows from the first hypothesis in characteristic $\ne 2$). Deligne \cite{DeligneLetter3} has then shown that $\eta(A)=1$.
\end{remark}

\subsection{The comparison theorem}

The following theorem explains the precise relationship between S-algebras and Azumaya algebras.

\begin{theorem} \label{thm:azu-T}
Let $A$ be an algebra in $\fC$.
\begin{enumerate}
\item If $A$ is Azumaya with $\tilde{\eta}(A)=1$ and ${}_A\Mod_A \cong \fC$, then $A$ admits an S-structure.
\item If $A$ admits an S-structure then $A$ is Azumaya with $\tilde{\eta}(A)=1$.
\end{enumerate}
\end{theorem}

\begin{proof}
(a) Since $\tilde{\eta}(A)=1$, we have an isomorphism
\begin{displaymath}
i \colon A \otimes A \to (A \otimes A)^{\dag}
\end{displaymath}
of $A \otimes A$ bimodules that squares to the identity. Since the left module structure of the target is the usual one, we have $i=\rho_{\pi}$ for some $\pi \in \Gamma(A \otimes A)$ satisfying $\pi^2=1$. Since $\rho_{\pi}$ is twisted right linear, we have
\begin{displaymath}
\rho_{\pi}(x \otimes y)=\rho_{\pi}(1) (y \otimes x) = \lambda_{\pi} \tau(x \otimes y),
\end{displaymath}
and so $\rho_{\pi} =\lambda_{\pi} \tau$. Since $\pi^2=1$, we have $\tau=\lambda_{\pi} \rho_{\pi}$.

Consider the map $\phi \colon A \otimes A \to A$ defined by $\phi = \mu \circ \rho_{\pi}$. We have
\begin{displaymath}
\phi(axb \otimes y) = \mu((axb \otimes y)\pi) = \mu((a \otimes 1)(x \otimes y) \pi (1 \otimes b))=a \mu((x \otimes y)\pi) b = a \phi(x \otimes y) b.
\end{displaymath}
We thus see that $\phi$ is a map of $A$-bimodules, where the module structure on $A \otimes A$ comes from the first factor. From the equivalence ${}_A\Mod_A=\fC$, we see that $\phi=\id \otimes \epsilon$ for some map $\epsilon \colon A \to \bone$. Thus $\mu \rho_{\pi}=\id \otimes \epsilon$, and so $\mu=(\id \otimes \epsilon) \rho_{\pi}$. Thus $(\pi, \epsilon)$ is a S-structure.

(b) Now suppose that $(\pi, \epsilon)$ is a S-structure on $A$. To show that $A$ is Azumaya, we must show that the map
\begin{displaymath}
\phi \colon A \otimes A \to \uHom(A) = A \otimes A^*
\end{displaymath}
is an isomorphism. We identify $A^*$ with $A$ using the self-duality of $A$ provided by $\epsilon \mu$. With this identification, we compute $\phi$, and find it is given by
\begin{displaymath}
\phi(x \otimes y) = (\id \otimes \epsilon \otimes \id \otimes \epsilon)((x \otimes y \otimes 1 \otimes 1) \pi_1 \pi_2 \pi_3).
\end{displaymath}
To compute the above map, we first apply $\id \otimes \id \otimes \id \otimes \epsilon$. We note that $x \otimes y \otimes 1 \otimes 1$ and $\pi_1$ pull out of this map. We have
\begin{displaymath}
(\id \otimes \id \otimes \id \otimes \epsilon)(\pi_2 \pi_3) = (\id \otimes \id \otimes \mu)(\pi_2) = \pi_2.
\end{displaymath}
We thus find
\begin{align*}
\phi(x \otimes y)
&= (\id \otimes \epsilon \otimes \id)((x \otimes y \otimes 1)\pi_1 \pi_2) \\
&= (\id \otimes \id \otimes \epsilon)((x \otimes 1 \otimes y)\pi_2 \pi_1) \\
&= (x \otimes 1) \cdot \big( (\id \otimes \id \otimes \epsilon)((1 \otimes 1 \otimes y)\pi_2) \big) \cdot \pi \\
&= (x \otimes y) \pi.
\end{align*}
In the third step we used the formula $(\id \otimes \epsilon)\rho_{\pi} = \mu$. We thus find $\phi=\rho_{\pi}$, which is indeed an isomorphism, and so $A$ is Azumaya. The map $\rho_{\pi}$ provides an isomorphism of bimodules $A \otimes A \to (A \otimes A)^{\dag}$, which squares to the identity, and so $\tilde{\eta}(A)=1$.
\end{proof}

We now use general properties of Azumaya algebras to easily prove two results on S-algebras. For the first result, we must assume that the functor $\fC \to {}_A\Mod_A$ is an equivalence.

\begin{proposition} \label{prop:Sclass}
If $(\pi, \epsilon)$ is an S-structure on $A$ then $(-\pi, -\epsilon)$ is the only other one.
\end{proposition}

\begin{proof}
Suppose $(\pi', \epsilon')$ is an arbitrary S-structure. Then $\rho_{\pi} \rho_{\pi'}$ is an automorphism of $A \otimes A$ as an $A \otimes A$ bimodule, and is thus a scalar via the equivalence $\fC \cong {}_A\Mod_A$. Hence $\pi'= \alpha \pi$ for some scalar $\alpha$. Since $(\pi')^2=\pi^2=1$, it follows that $\alpha=\pm 1$. Clearly, $\epsilon'=\alpha \epsilon$ as well.
\end{proof}

\begin{proposition} \label{prop:anti-inv}
Let $A$ be an S-algebra and let $i \colon A \to A$ be an anti-involution preserving $\epsilon$. Then $A^{\otimes 2}$ is S-split.
\end{proposition}

\begin{proof}
Since $A$ is an Azumaya algebra, the natural map $A \otimes A^{\op} \to \uEnd(A)$ is an isomorphism, and one easily verifies that it is in fact an S-splitting. The anti-involution $i$ defines an algebra isomorphism $A \to A^{\op}$, which necessarily preserves the S-structure since it preserves $\epsilon$. We thus have an S-isomorphism $A^{\otimes 2} \to \uEnd(A)$, which completes the proof.
\end{proof}

\section{Splitting categories} \label{s:split}

\subsection{Statement of results}

Fix a Karoubian tensor category $\fC$ and a separable S-algebra $A$ in $\fC$ for the duration of \S \ref{s:split}. We recall (Proposition~\ref{prop:separable}) that separability is automatic if $\fC$ is semi-simple or $\udim(A) \ne 0$. Much of what we do goes through without separability; see \S \ref{ss:insep} for details. Our aim is to enlarge $\fC$ in a way to make $A$ be S-split. To precisely discuss this idea, we introduce the following notion.

\begin{definition}
A \defn{splitting category} for $A$ is a tuple $(\fD, \Phi, X, i)$, where $\fD$ is a Karoubian tensor category, $\Phi \colon \fC \to \fD$ is a tensor functor, $X$ is a rigid object of $\fD$, and $i \colon \Phi(A) \to \uEnd(X)$ is an isomorphism of S-algebras.
\end{definition}

We note that in any splitting category, we have $\udim(X)=\udeg(A)$; indeed, the tensor functor $\Phi$ preserves categorical degree, and the categorical degree of $\uEnd(X)$ is $\udim(X)$. There may be many splitting categories for $A$; indeed if $A$ is already split in $\fC$ then any tensor functor $\Phi \colon \fC \to \fD$ is naturally a splitting category. We therefore aim to single out a best one. The following definition gives the relevant notion.

\begin{definition} \label{defn:usplit}
A splitting category $(\fD, \Phi, X, i)$ is \defn{universal} if it is an initial object of the (1-truncation of the) 2-category of splitting categories, in the following sense. Suppose that $(\fT, \Psi, Y, j)$ is an arbitrary splitting category. Then there must be a tensor functor $\Theta \colon \fD \to \fT$ and isomorphisms $a \colon \Psi \to \Theta \circ \Phi$ and $b \colon Y \to \Theta(X)$ such that the diagram
\begin{displaymath}
\xymatrix@C=3em{
\Psi(A) \ar[r]^-j \ar[d]_a  & \uEnd(Y) \ar[d]^b \\
\Theta(\Phi(A)) \ar[r]^-{\Theta(i)} & \uEnd(\Theta(X)) }
\end{displaymath}
commutes. Moreover, $(\Theta, a, b)$ must be unique up to isomorphism.
\end{definition}

A universal splitting category is unique up to equivalence, if it exists. The following is one of the main results of this paper.

\begin{theorem} \label{thm:split}
A universal splitting category $(\fD, \Phi, X, i)$ for $A$ exists. Additionally:
\begin{enumerate}
\item The category $\fD$ carries a natural $\bZ$-grading, with $X \in \fD_1$. The functor $\Phi$ is an equivalence between $\fC$ and $\fD_0$; in particular, $\Phi$ is fully faithful.
\item If $A=\uEnd(V)$ is S-split then $\fD=\fC^{\bZ}$ is the category of $\bZ$-graded objects in $\fC$ and $X=V[1]$; see \S \ref{ss:split-case} for details.
\item If $\fC$ is rigid (resp.\ $\Hom$-finite, abelian, semi-simple, pre-Tannakian), so is $\fD$.
\end{enumerate}
\end{theorem}

As a corollary of the theorem, we can describe the splitting behavior of Azumaya algebras.

\begin{corollary}
Let $B$ be an Azumaya algebra in $\fC$. Then the following are equivalent:
\begin{enumerate}
\item There exists a fully faithful tensor functor $\Phi \colon \fC \to \fD$, for some tensor category $\cD$, such that $\Phi(B)$ is split.
\item We have $\tilde{\eta}(B)=1$.
\end{enumerate}
\end{corollary}

\begin{proof}
If (b) holds then we can choose an S-structure on $B$ (Theorem~\ref{thm:azu-T}), and take $\fD$ to be the universal splitting category. (If $B$ is not separable, the proof of Theorem~\ref{thm:split} still produces a $\fD$ that works.) Now suppose (a) holds. Since $\Phi(B)$ is S-split, we have
\begin{displaymath}
\Phi(\tilde{\eta}(B)) = \tilde{\eta}(\Phi(B)) = 1.
\end{displaymath}
Since $\Phi$ is fully faithful, it is injective on $\wt{\Pic}(-)[2]$, and so $\tilde{\eta}(B)=1$.
\end{proof}

We can also use the theorem to give a useful intrinsic characterization of the universal splitting category. In what follows, for a rigid object $X$ and an integer $r \le 0$, we let $X^{\otimes r}$ denote $(X^*)^{\otimes (-r)}$.

\begin{corollary} \label{cor:uni-char}
Let $(\fT, \Psi, Y, j)$ be a splitting category. Suppose that:
\begin{enumerate}
\item $\fT$ is $\bZ$-graded, and $Y \in \fT_1$.
\item $\Psi \colon \fC \to \fT_0$ is an equivalence.
\item Every object of $\fT$ is a summand of $\Psi(M) \otimes Y^{\otimes r}$ for some $M \in \fC$ and $r \in \bZ$.
\end{enumerate}
Then $(\fT, \Psi, Y, j)$ is the universal splitting category.
\end{corollary}

\begin{proof}
Let $(\fD, \Phi, X, i)$ be the universal splitting category, and let $\Theta \colon \fD \to \fT$ be the tensor functor afforded by universality. It is clear that $\Theta$ is compatible with the $\bZ$-gradings. Let $M,N \in \fC$, and let $r,s \in \bZ$. We have identifications
\begin{align*}
\Hom_{\fD}(\Phi(M) \otimes X^{\otimes r}, \Phi(N) \otimes X^{\otimes r})
&= \Hom_{\fD}(\Phi(M) \otimes \Phi(A)^{\otimes r}, \Phi(N)) \\
&= \Hom_{\fC}(M \otimes A^{\otimes r}, N),
\end{align*}
where in the first step we used the rigidity of $X$, and in the second that $\Phi$ is fully faithful. A similar computation holds in $\fT$. We thus see that the natural map
\begin{displaymath}
\Theta \colon \Hom_{\fD}(\Phi(M) \otimes X^{\otimes r}, \Phi(N) \otimes X^{\otimes s}) \to
\Hom_{\fT}(\Psi(M) \otimes Y^{\otimes r}, \Psi(N) \otimes Y^{\otimes s}).
\end{displaymath}
is an isomorphism. Indeed, if $r=s$ then $\Theta$ is compatible with the previous identifications, and if $r \ne s$ then both $\Hom$ spaces vanish due to the grading. Since all objects are summands of ones of the above kind, we see that $\Theta$ is fully faithful. Assumption~(c) implies that $\Theta$ is essentially surjective, and so the result follows.
\end{proof}

\subsection{Preliminaries} \label{ss:prelim}

We now begin working on the proof of Theorem~\ref{thm:split}, which will take the remainder of \S \ref{s:split}. For $n,m \in \bN$, put
\begin{displaymath}
A_{n,m} = A^{\otimes n} \otimes (A^{\op})^{\otimes m}.
\end{displaymath}
For some of our constructions, it will be clearer to work with finite sets than natural numbers. A \defn{biset} is a finite set $x$ with a given decomposition $x=x_+ \sqcup x_-$. Given a biset $x$, we put
\begin{displaymath}
A_x = A^{\otimes x_+} \otimes (A^{\op})^{\otimes x_-}.
\end{displaymath}
The \defn{standard biset} $[n,m]$ is $[n] \amalg [m]$, where $n,m \in \bN^2$. Of course, if $x=[n,m]$ then $A_x=A_{n,m}$. An \defn{isomorphism} of bisets $\sigma \colon x \to y$ is a bijection respecting the decompositions. Such a map induces an algebra isomorphism $\sigma_* \colon A_x \to A_y$. In particular, if $M$ is an $A_y$-module then we can define an $A_x$-module structure on $M$ via $\sigma_*$; we denote this by $\sigma^*(M)$.

Write $\Aut(x)$ for the automorphism group of the biset $x$. We have a natural group homomorphism
\begin{displaymath}
\pi \colon \Aut(x) \to \Gamma(A_x)^{\times},
\end{displaymath}
just as in the discussion following Proposition~\ref{prop:braid}. We note that for $\sigma, \sigma' \in \Aut(x)$, we have $\sigma_*(\pi(\sigma'))=\pi(\sigma \sigma' \sigma^{-1})$.

\subsection{The category $\tilde{\fD}$} \label{ss:tildeD}

As a first step towards constructing $\fD$, we construct an auxiliary tensor category $\tilde{\fD}$. The underlying category is simple enough:
\begin{displaymath}
\tilde{\fD} = \bigoplus_{n,m \in \bN^2} \tilde{\fD}_{n,m}, \qquad \tilde{\fD}_{n,m} = \Mod^{\rm relp}_{A_{n,m}}.
\end{displaymath}
Here $\Mod^{\rm relp}$ is the category of relatively projective modules. Thus an object $M$ of $\tilde{\fD}$ is a collection $(M_{n,m})_{n,m}$ where $M_{n,m}$ is a relatively projective $A_{n,m}$-module, and $M_{n,m}=0$ for all but finitely many pairs $(n,m)$. We will find it convenient to sometimes index $M$ by standard bisets; that is, for $x=[n,m]$, we write $M_x$ for $M_{n,m}$.

The tensor product $\otimes$ on $\tilde{\fD}$ is also easy to describe:
\begin{displaymath}
(M \otimes N)_{n,m} = \bigoplus_{0 \le i \le n, 0 \le j \le m} M_{i,j} \otimes N_{n-i,m-j}.
\end{displaymath}
Here $M_{i,j} \otimes N_{n-i,m-j}$ is naturally a module for $A_{i,j} \otimes A_{n-i,m-j}$, which we identify with $A_{n,m}$. This clearly defines a monoidal structure on $\tilde{\fD}$. However, the symmetry of $\otimes$ is not obvious.

We look at a simple case to illustrate the subtlety. Let $M$ and $N$ be $A$-modules, regarded as objects of $\tilde{\fD}$ in degree $(1,0)$. We must construct an isomorphism
\begin{displaymath}
\tilde{\tau} \colon M \otimes N \to N \otimes M
\end{displaymath}
of objects of $\tilde{\fD}$, i.e., of $A^{\otimes 2}$-modules. We have an isomorphism $\tau=\tau_{M,N}$ coming from the symmetry of the tensor product in $\fC$. However, $\tau$ is not $A^{\otimes 2}$-linear; it satisfies the equation
\begin{displaymath}
\tau((a \otimes b)(m \otimes n))=(b \otimes a) \tau(m \otimes n)
\end{displaymath}
for $a,b \in A$ and $m \in M$ and $n \in N$. We can fix this issue by multiplying by $\pi$; that is, we define $\tilde{\tau} =\lambda_{\pi} \circ \tau$. This gives the required isomorphism.

To describe the symmetry in general, let $\sigma_{n,m}$ be the element of the symmetry group $\fS_{n+m}$ inducing the bijection $[n] \amalg [m] \to [m] \amalg [n]$. If $M \in \tilde{\fD}_{m,m'}$ and $N \in \tilde{\fD}_{n,n'}$ then the symmetry isomorphism is
\begin{displaymath}
\tilde{\tau}_{M,N} = \rho_{\pi(\sigma_{m',n'})} \lambda_{\pi(\sigma_{m,n})} \tau_{M,N},
\end{displaymath}
where here we write $\rho$ for the $(A^{\op})^{\otimes (n'+m')}$-module structure, and $\pi(-)$ is defined after Proposition~\ref{prop:braid}. Verifying the hexagon axioms from this definition is somewhat complicated. We therefore take a different approach: we define a new tensor product $\boxtimes$ on $\tilde{\fD}$ that ends up being equivalent to $\otimes$, but for which the symmetry will be more evident.

Before getting to $\boxtimes$, we make one more observation. Suppose $M$ is an object of $\fC$. Then there is an associated object $\tilde{M}$ of $\tilde{\fD}$ defined by $\tilde{M}_{0,0}=M$ and $\tilde{M}_{n,m}=0$ for $(n,m) \ne (0,0)$. This construction defines a symmetric monoidal functor $\tilde{\Phi} \colon \fC \to \tilde{\fD}$.

\subsection{A second tensor product}

Suppose $M$ and $N$ are objects of $\tilde{\fD}$. For a biset $x$, consider the following direct sum
\begin{displaymath}
Q = \bigoplus_{\gamma \colon x \to y \amalg z} \gamma^*(M_y \otimes N_z).
\end{displaymath}
Here $y$ and $z$ vary over all standard bisets, $\gamma$ varies over all isomorphisms of bisets, and the tensor product is taken in $\fC$. The tensor product $M_y \otimes N_z$ is naturally a module for $A_y \otimes A_z=A(y \amalg z)$, and so $\gamma^*(M_y \otimes N_z)$ is an $A_x$-module. Thus the entire direct sum above is an $A_x$-module. Concretely, if $a \in A_x$ and $q \in Q$ then $(aq)_{\gamma}=\gamma_*(a) q_{\gamma}$; here $q_{\gamma}$ is the $\gamma$ component of $q$. Although we are using elemental notation here, there are no actual elements; this is simply a convenient way of specifying the action map.

For $\sigma \in \Aut(x)$ we define a map $\sigma \colon Q \to Q$ by
\begin{displaymath}
(\sigma q)_{\gamma} = \gamma_*(\pi(\sigma)^{-1}) \cdot q_{\gamma \sigma}.
\end{displaymath}
Here $\pi(\sigma)^{-1}$ belongs to $\Gamma(A_x)$, and applying $\gamma_*$ to it moves it to $A(y \amalg z)$. Suppose $\sigma'$ is a second automorphism of $x$. Then
\begin{align*}
(\sigma (\sigma' q))_{\gamma}
&= \gamma_*(\pi(\sigma)^{-1}) (\sigma' q)_{\gamma \sigma}
= \gamma_*(\pi(\sigma)^{-1}) (\gamma_* \sigma_*)(\pi(\sigma')^{-1}) q_{\gamma \sigma \sigma'} \\
&= \gamma_*(\pi(\sigma)^{-1} \pi(\sigma \sigma' \sigma^{-1})^{-1}) q_{\gamma \sigma \sigma'}
= \gamma_*(\pi(\sigma \sigma')^{-1}) q_{\gamma \sigma \sigma'} = ((\sigma \sigma') q)_{\gamma}
\end{align*}
We thus see that we have an action of the group $\Aut(x)$. For $a \in A_x$, we have
\begin{align*}
(\sigma (aq))_{\gamma}
&= \gamma_*(\pi(\sigma)^{-1}) (aq)_{\gamma \sigma}
= \gamma_*(\pi(\sigma)^{-1}) (\gamma \sigma)_*(a) q_{\gamma \sigma}
= \gamma_*(\pi(\sigma)^{-1} \sigma_*(a)) q_{\gamma \sigma} \\
&= \gamma_*(a \pi(\sigma)^{-1}) q_{\gamma \sigma}
= \gamma_*(a) (\sigma q)_{\gamma}
= (a (\sigma q))_{\gamma}
\end{align*}
Here we used the formula $\sigma_*(a)=\pi(\sigma) a \pi(\sigma^{-1})$. We thus see that $\Aut(\sigma)$ acts on $Q$ by $A_x$-module automorphisms.

We now define an object $M \boxtimes N$ of $\tilde{\fD}$ by
\begin{displaymath}
(M \boxtimes N)_x = \big( \bigoplus_{\gamma \colon x \to y \amalg z} \gamma^*(M_y \otimes N_z) \big)^{\Aut(x)},
\end{displaymath}
where here $x$ is a standard biset, the superscript denotes invariants, and other notation is as above. Since $\Aut(x)$ freely permutes the summands, the invariant space exists: it is isomorphic to the direct sum of orbit representatives of summands. Since $\Aut(x)$ acts by $A_x$-module homomorphisms, $(M \boxtimes N)_x$ is an $A_x$-module. It is clear that $\boxtimes$ is functorial.

As we said above, $(M \boxtimes N)_x$ is isomorphic to the direct sum of the summands as $\gamma$ varies over orbit representatives. There is a convenient set of representatives to use, namely, those for which $\gamma$ is order-preserving on the $+$ and $-$ pieces of each biset. Here we use the standard order on $[n]$, and order $[n] \amalg [m]$ by putting $[n]$ before $[m]$. In this way, we obtain an isomorphism
\begin{displaymath}
M \boxtimes N \to M \otimes N.
\end{displaymath}
This map simply projects onto the components where $\gamma$ is order-preserving. This isomorphism is natural in the technical sense, i.e., it is an isomorphism of functors. However, it is unnatural in a colloquial sense, as it relies on a preferred order on the set $[n] \amalg [m]$, and this is why the symmetry is harder to see on $\otimes$.

We have a natural isomorphism
\begin{displaymath}
\tilde{\tau} \colon M \boxtimes N \to N \boxtimes M
\end{displaymath}
by simply permuting the summands around. Precisely, write $(M \boxtimes N)_{x,\gamma}$ for the $\gamma$ summand of the tensor product. For an isomorphism $\gamma \colon x \to y \amalg z$, let $\gamma' \colon x \to z \amalg y$ be the induced isomorphism. Then the symmetry $\tau$ in $\fC$ gives an isomorphism
\begin{displaymath}
\tau \colon (M \boxtimes N)_{x, \gamma} \to (N \boxtimes M)_{x, \gamma'}.
\end{displaymath}
The isomorphism $\tilde{\tau}$ is induced from these isomorphisms. One easily sees that it is $A_x$-linear.

We have just defined the binary product $\boxtimes$. More generally, for a finite set $I$, we define a tensor product
\begin{displaymath}
\boxtimes \colon \tilde{\fD}^I \to \tilde{\fD}
\end{displaymath}
by
\begin{displaymath}
(\boxtimes_{i \in I} M_i)_x = \big( \bigoplus_{\gamma \colon x \to \amalg_{i \in I} y_i} \gamma^*(\bigotimes_{i \in I} M_{i,y_i}) \big)^{\Aut(x)}.
\end{displaymath}
This construction is naturally equivariant for the symmetric group $\Aut(I)$. Moreover, if $f \colon I \to J$ is a surjection of finite sets and $\{M_i\}_{i \in I}$ is a family of objects in $\tilde{\fD}$, then there is a natural isomorphism
\begin{displaymath}
\boxtimes_{i \in I} M_i \cong \boxtimes_{j \in J} \boxtimes_{i \in f^{-1}(j)} M_i.
\end{displaymath}
The above isomorphisms satisfy the expected compatibilities. This shows that our binary tensor product $\boxtimes$ does indeed define a symmetric monoidal structure. Therefore $\otimes$ does as well.

\subsection{The category $\fD$}

Let $E$ be the object of $\tilde{\fD}$ defined by $E_{1,1}=A$ (with its natural $A_{1,1}$ structure) and $E_{n,m}=0$ for $(n,m) \ne (1,1)$. We require the following observation:

\begin{lemma}
The symmetry $\tilde{\tau}$ on $E^{\otimes 2}$ is trivial.
\end{lemma}

\begin{proof}
If $M$ and $N$ are objects of $\tilde{\fD}_{1,1}$ then the symmetry
\begin{displaymath}
\tilde{\tau} \colon M \otimes N \to N \otimes M
\end{displaymath}
is $\rho_{\pi} \lambda_{\pi} \tau$, where $\tau$ is the symmetry from $\fC$, and $\rho_{\pi}$ denotes the $A^{\op}$-module structure. We specialize this to the case $M=N=E$. Since $\rho_{\pi} \lambda_{\pi} \tau$ is trivial on $A$, by the definition of S-structure, the result follows.
\end{proof}

We now define $\fD=\tilde{\fD}\{E\}$ to be the trivialization of $E$, as defined in \S \ref{ss:triv2}. Recall that there is a canonical functor $\tilde{\fD} \to \fD$, whose defining property is that $E$ is sent to the monoidal unit $\bone$. We have a natural symmetric monoidal functor $\Phi \colon \fC \to \fD$ by composing $\tilde{\Phi}$ with the canonical functor $\tilde{\fD} \to \fD$.

The category $\tilde{\fD}$ is bi-graded; we coarsen this to a $\bZ$-grading by placing $\tilde{\fD}_{n,m}$ in degree $m-n$. Note the sign here: using $n-m$ would be a more obvious choice, but our sign convention is necessary to ensure that the splitting object $X$ has degree $+1$. The object $E$ has degree~0, and so the $\bZ$-grading on $\tilde{\fD}$ induces one on $\fD$.

\begin{lemma} \label{lem:D-equiv}
The natural functor $\tilde{\fD}_{n,m} \to \fD_{m-n}$ is an equivalence.
\end{lemma}

\begin{proof}
Write $\Pi \colon \tilde{\fD} \to \fD$ for the natural functor. By definition, this factors as $\Pi_2 \circ \Pi_1$, where
\begin{displaymath}
\Pi_1 \colon \tilde{\fD} \to \tilde{\fD}[1/E], \qquad
\Pi_2 \colon \tilde{\fD}[1/E] \to \tilde{\fD}\{E\} = \fD
\end{displaymath}
are the natural functors. Let $M$ and $N$ be objects of $\tilde{\fD}_{n,m}$ and $\tilde{\fD}_{a,b}$. As $\Pi_1(M)=(M,0)$, and similarly for $N$, we have
\begin{displaymath}
\Hom_{\tilde{\fD}[1/E]}(\Pi_1(M), \Pi_1(N)) = \varinjlim \Hom_{\tilde{\fD}}(E^{\otimes r} \otimes M, E^{\otimes r} \otimes N).
\end{displaymath}
If $(n,m) \ne (a,b)$ then $E^{\otimes r} \otimes M$ and $E^{\otimes r} \otimes N$ belong to different homogeneous pieces of $\tilde{\fD}$, and so all $\Hom$ spaces on the right vanish; thus the left side vanishes as well. Now suppose $(n,m)=(a,b)$. The functor
\begin{displaymath}
- \otimes E \colon \tilde{\fD}_{n,m} \to \tilde{\fD}_{n+1,m+1}
\end{displaymath}
is an equivalence, since $A$ is an Azumaya algebra, and in particular fully faithful. Thus all transition maps in the direct limit are isomorphisms, and so we see that the natural map
\begin{displaymath}
\Hom_{\tilde{\fD}}(M, N) \to \Hom_{\tilde{\fD}[1/E]}(\Pi_1(M), \Pi_1(N))
\end{displaymath}
is an isomorphism.

Maintain the assumption $(n,m)=(a,b)$. We have
\begin{displaymath}
\Hom_{\fD}(\Pi(M), \Pi(N)) = \bigoplus_{\ell \in \bZ} \Hom_{\tilde{\fD}[1/E]}(\Pi_1(M), \Pi_1(E)^{\otimes \ell} \otimes \Pi_1(N)).
\end{displaymath}
Since $\Pi_1$ is a tensor functor, the second argument is $\Pi_1(E^{\otimes \ell} \otimes N)$. As $E^{\otimes \ell} \otimes N$ has degree $(\ell+n, \ell+m)$, the above $\Hom$ space vanishes unless $\ell=0$. Thus, combined with the previous paragraph, we see that the natural map
\begin{displaymath}
\Hom_{\tilde{\fD}}(M, N) \to \Hom_{\fD}(\Pi(M), \Pi(N))
\end{displaymath}
is an isomorphism. Thus $\Pi \colon \tilde{\fD}_{n,m} \to \fD_{m-n}$ is fully faithful.

Finally, we must show essential surjectivity. Every object of $\tilde{\fD}[1/E]$ is isomorphic to $\Pi_1(M) \otimes \Pi_1(E)^{\otimes \ell}$ for some object $M$ of $\tilde{\fD}$ and some $\ell \in \bZ$. Also, $\Pi_2$ is bijective on objects. Thus every object of $\fD$ has the form $\Pi(M) \otimes \Pi(E)^{\otimes \ell}$. Since $\Pi(E) \cong \bone$, it follows that every object of $\fD$ is isomorphic to one of the form $\Pi(M)$ with $M$ in $\tilde{\fD}$. Clearly, any object of $\fD_a$ is isomorphic to some $\Pi(M)$ with $M \in \tilde{\fD}_{n,m}$ and $m-n=a$. It thus suffices to show that if $n-m=a-b$ then the essential images of $\tilde{\fD}_{n,m}$ and $\tilde{\fD}_{a,b}$ in $\fD$ are the same. We immediately reduce to the case $(a,b)=(n+1,m+1)$. The result now follows since these two categories are equivalent via $- \otimes E$.
\end{proof}

For an integer $n$, put $A_n=A_{0,n}$ if $n \ge 0$ and $A_n=A_{-n,0}$ if $n \le 0$. Again, we emphasize that $A=A_{-1}$ and $A^{\op}=A_1$, which is a bit counterintuitive. From the above lemma, we see that $\fD_n$ is equivalent to the category of relatively projective $A_n$-modules.

We now briefly describe how the tensor product in $\fD$ interacts with the above identifications. For simplicity, we just look at one representative case. Suppose $M$ is an $A_2$-module, regarded in either $\tilde{\fD}_{0,2}$ or $\fD_2$, and $N$ is an $A_{-1}$-module, regarded in either $\fD_{1,0}$ or $\fD_{-1}$. The tensor product $M \otimes N$ in $\tilde{\fD}$ is just the tensor product in $\fC$, equipped with its natural $A_{1,2}$-structure. This defines an object in $\fD_1$, which we have identified with the category of $A_1$-modules. To go from $M \otimes N$ to an $A_1$-module, we apply the Morita equivalence $\Mod_{A_{1,2}}=\Mod_{A_1}$.

\subsection{The splitting structure}

We have so far constructed a tensor functor $\Phi \colon \fC \to \fD$. We now show that $\fD$ admits the structure of a splitting category for $A$. For an integer $n$, let $X_n=A_n$, regarded as an object of $\fD_n$, or as an object of $\tilde{\fD}_{0,n}$ (for $n \ge 0$) or $\tilde{\fD}_{-n,0}$ (for $n \le 0$). We also let $X=X_1$ and $Y=X_{-1}$. We note that for $n \ge 0$ we have $X_n \cong X^{\otimes n}$ and $X_{-n} = Y^{\otimes n}$, in either $\tilde{\fD}$ or $\fD$.

\begin{lemma} \label{lem:Xrigid}
$X$ is a rigid object of $\fD$ with dual $Y$.
\end{lemma}

\begin{proof}
We work in $\tilde{\fD}$ in this paragraph. Define maps
\begin{align*}
\alpha \colon E &\to X \otimes Y, &
\beta \colon Y \otimes X &\to E \\
x &\mapsto (1 \otimes x)\pi, &
x \otimes y &\mapsto (\id \otimes \epsilon)((x \otimes y)\pi).
\end{align*}
Note that $\beta$ is simply the multiplication map $\mu$. One readily verifies that these are maps in $\tilde{\fD}$. We claim that the composition
\begin{displaymath}
\xymatrix{
E \otimes X \ar[r]^-{\alpha \otimes \id} &
X \otimes Y \otimes X \ar[r]^-{\id \otimes \beta} &
X \otimes E }
\end{displaymath}
is the symmetry isomorphism. Indeed, we have
\begin{align*}
(\id \otimes \beta)(\alpha \otimes \id)(a \otimes b)
&= (\id \otimes \id \otimes \epsilon)((1 \otimes a \otimes b) \pi_1 \pi_2) \\
&= (\id \otimes \id \otimes \epsilon)(\pi_1 (a \otimes 1 \otimes b) \pi_2) \\
&= \pi (a \otimes b) = (b \otimes a) \pi.
\end{align*}
Here $\pi_1$ and $\pi_2$ are as before Proposition~\ref{prop:braid}. In the first step, we simply applied the definitions; in the second, we used that conjugation by $\pi_1$ transposes the first two tensor factors; in the third, we use the relationship between $\mu$ and $(\pi, \epsilon)$, and the fact that $\pi_1$ pulls out of the map $\id \otimes \id \otimes \epsilon$; in the final step, we again used that conjugation by $\pi$ transposes tensor factors. This computation verifies the claim, as the final expression is indeed the explicit formula for the symmetry isomorphism in this case.

We now apply the natural functor $\tilde{\fD} \to \fD$. We see that (the images of) $\alpha$ and $\beta$ define co-evaluation and evaluation maps between $X$ and $Y$. The above computation verifies the first zig-zag axiom (see Definition~\ref{defn:rigid}), and the second one follows similarly.
\end{proof}

\begin{corollary}
$X_n$ is a rigid object of $\fD$ with dual $X_{-n}$, for any $n \in \bZ$.
\end{corollary}

Since $X$ is a rigid object of $\fD$ in degree~1, it follows that $\uEnd(X)=X \otimes X^*$ is an S-algebra in $\fD_0=\fC$.

\begin{lemma}
We have a natural isomorphism of S-algebras $i \colon A \to \uEnd(X)$ in $\fC$.
\end{lemma}

\begin{proof}
We have an isomorphism in $\tilde{\fD}$
\begin{displaymath}
j \colon A \otimes E \to X \otimes Y
\end{displaymath}
given by $j=\rho_{\pi}$; here $E$, $A$, $X$, $Y$ are each copies of $A$, but placed in degrees $(1,1)$, $(0,0)$, $(0,1)$, and $(1,0)$. The isomorphism $j$ induces an isomorphism $i \colon A \to X \otimes Y$ in $\fD$. We must now show that $i$ is an algebra isomorphism.

Consider the diagram in $\tilde{\fD}$
\begin{displaymath}
\xymatrix@C=3em{
E \ar[d]_{1 \otimes \id} \ar[rd]^{\alpha} \\
A \otimes E \ar[r]^-j & X \otimes Y }
\end{displaymath}
where 1 is the unit of $A$, and $\alpha$ is the map from the proof Lemma~\ref{lem:Xrigid}. The diagram clearly commutes. Passing to $\fD$, we see that $i$ is compatible with units.

Consider the diagram in $\tilde{\fD}$
\begin{displaymath}
\xymatrix@C=4em{
A \otimes E \otimes A \otimes E \ar[d]_{j \otimes j} \ar[r]^-{\id \otimes \tilde{\tau} \otimes \id} &
A \otimes A \otimes E \otimes E \ar[r]^-{\mu \otimes \id \otimes \id} &
A \otimes E \otimes E \ar[d]^{j \otimes \id} \\
X \otimes Y \otimes X \otimes Y \ar[r]^-{\id \otimes \beta \otimes \id} &
X \otimes E \otimes Y \ar[r]^-{\id \otimes \tilde{\tau}} &
X \otimes Y \otimes E }
\end{displaymath}
where $\tilde{\tau}$ is the symmetry in $\tilde{\fD}$, and $\beta$ is from the proof of Lemma~\ref{lem:Xrigid}. We show that this diagram commutes; this will prove that $i$ is compatible with multiplication. Call the top path in the diagram $\phi$ and the bottom path $\psi$. Since $A$ is in degree $(0,0)$, the symmetry $\tilde{\tau}$ on $A \otimes E$ is simply the symmetry $\tau$ from $\fC$. We therefore have
\begin{displaymath}
\phi(a_1 \otimes b_1 \otimes a_2 \otimes b_2) = (a_1 a_2 \otimes b_1 \otimes b_2) \pi_1.
\end{displaymath}
Let $\psi'$ be the composition of the first two maps defining $\psi$. We have
\begin{displaymath}
\psi'(a_1 \otimes b_1 \otimes a_2 \otimes b_2) = (\id \otimes \id \otimes \epsilon \otimes \id)((a_1 \otimes b_1 \otimes a_2 \otimes b_2)\pi_1 \pi_3 \pi_2).
\end{displaymath}
Now, it will be somewhat clearer to have the $\epsilon$ on the final tensor factor; we can arrange this by conjugating the argument by $\pi_3$. We thus have
\begin{displaymath}
\psi'(a_1 \otimes b_1 \otimes a_2 \otimes b_2) = (\id \otimes \id \otimes \id \otimes \epsilon)(\pi_3 (a_1 \otimes b_1 \otimes a_2 \otimes b_2)\pi_1 \pi_3 \pi_2 \pi_3).
\end{displaymath}
We now move the first three $\pi$'s on the right to the left, permuting the tensor factors:
\begin{displaymath}
\psi'(a_1 \otimes b_1 \otimes a_2 \otimes b_2) = (\id \otimes \id \otimes \id \otimes \epsilon)(\pi_1 \pi_2 (b_1 \otimes b_2 \otimes a_1 \otimes a_2) \pi_3).
\end{displaymath}
Now, $\pi_1 \pi_2$ pulls out of the function, and what remains is the multiplication map on the 3rd and 4th factors. We thus find
\begin{displaymath}
\psi'(a_1 \otimes b_1 \otimes a_2 \otimes b_2) = \pi_1 \pi_2 (b_1 \otimes b_2 \otimes a_1 a_2).
\end{displaymath}
Now, to obtain $\psi$ we must apply $\id \otimes \tau'$. Since $E$ is in degree $(1,1)$ and $Y$ is in degree $(1,0)$, we have $\tilde{\tau} = \lambda_{\pi} \tau$ on $E \otimes Y$; for these particular modules, this becomes $\tilde{\tau}=\rho_{\pi}$. We thus have
\begin{displaymath}
\psi(a_1 \otimes b_1 \otimes a_2 \otimes b_2) = \pi_1 \pi_2 (b_1 \otimes b_2 \otimes a_1 a_2) \pi_2.
\end{displaymath}
Moving the $\pi$'s on the left to the right, we find $\psi=\phi$, as required.
\end{proof}

We have thus shown that $(\fD, \Phi, X, i)$ is a splitting category for $A$.

\subsection{The split case} \label{ss:split-case}

Suppose now that $A=\uEnd(V)$ is S-split, where $V$ is a rigid object of $\fC$. Let $V_1=V^*$ equipped with its canonical $A^{\op}$-module structure, and regarded as an object of $\tilde{\fD}_{0,1}$ or of $\fD_1$. Similarly, let $V_{-1}=V$ regarded in $\tilde{\fD}_{1,0}$ or $\fD_{-1}$.

\begin{lemma} \label{lem:Vinvert}
$V_1$ is an invertible object of $\fD$ with sign $+1$.
\end{lemma}

\begin{proof}
It follows immediately from the definitions that we have a natural isomorphism $V_1 \otimes V_{-1} = E$ in $\tilde{\fD}$, and thus an isomorphism $V_1 \otimes V_{-1} = \bone$ in $\fD$. We thus see that $V_1$ is an invertible object of $\fD$. The symmetry $\tilde{\tau}$ on $V_{-1}^{\otimes 2}=V^{\otimes 2}$ is the map $\lambda_{\pi} \circ \tau$, which is the identity by definition of S-split. Thus $V_{-1}$ has sign $+1$, and so $V_1$ does too.
\end{proof}

Let $\fC^{\bZ}$ denote the category of $\bZ$-graded objects in $\fC$, where all but finitely many graded components are zero. For an object $M$ of $\fC$, we let $M[n]$ denote the corresponding object of $\fC^{\bZ}$ concentrated in degree $n$. The category $\fC^{\bZ}$ carries a natural tensor product which makes it into a tensor category. Giving a tensor functor $\Theta \colon \fC^{\bZ} \to \fT$, for some tensor category $\fT$, is equivalent to giving a tensor functor $\Theta_0 \colon \fC \to \fT$ and an invertible object $L \in \fT$ of sign $+1$ (meaning the symmetry of $L^{\otimes 2}$ is trivial). Precisely, given $\Theta$, we take $\Theta_0$ to be its restriction to $\fC$, which is the degree~0 subcategory of $\fC^{\bZ}$, and we take $L=\Theta(\bone[1])$.

The above discussion shows that we have a natural tensor functor
\begin{displaymath}
\Theta \colon \fC^{\bZ} \to \fD
\end{displaymath}
that is the identity in degree~0, and satisfies $\Theta(\bone[1])=V_1$. We note that we have a natural isomorphism $X=\Theta(V[1])$. The following is the main result we are after:

\begin{lemma}
The above functor $\Theta$ is an equivalence.
\end{lemma}

\begin{proof}
The restriction of $\Theta$ to degree $n \ge 0$ is the functor $\fC \to \fD_n=\Mod_{A_n}$ given by $M \mapsto (V^*)^{\otimes n} \otimes M$, which is an equivalence by Morita theory. (The inverse equivalence is $N \mapsto V^{\otimes n} \otimes_{A_n} N$, which exists since $A_n$ is separable and $\fC$ is Karoubian.) The negative degree case is similar.
\end{proof}

\subsection{Functoriality} \label{ss:split-func}

Let $\Psi \colon \fC \to \fC'$ be a tensor functor, let $A'=\Psi(A)$, and let $\fD'$ be defined like $\fD$, but relative to $A'$. If $M$ is an $A_{n,m}$-module in $\fC$ then $\Psi(M)$ is naturally an $A'_{n,m}$-module in $\fC'$; moreover, if $M$ is relatively projective then so is $\Psi(M)$. It follows that $\Psi$ induces a functor
\begin{displaymath}
\tilde{\Psi} \colon \tilde{\fD} \to \tilde{\fD}'.
\end{displaymath}
It is clear from the construction that $\tilde{\Psi}$ is naturally a tensor functor. It is also clear that $\tilde{\Psi}$ maps $E$ to $E'$. It therefore induces a tensor functor
\begin{displaymath}
\Psi^+ \colon \fD \to \fD'.
\end{displaymath}
Thus the construction of the category $\fD$ is functorial. We note that the diagram
\begin{displaymath}
\xymatrix@C=3em{
\fD \ar[r]^{\Psi^+} & \fD' \\
\fC \ar[u]^{\Phi} \ar[r]^{\Psi} & \fC' \ar[u]_{\Phi'} }
\end{displaymath}
commutes (up to isomorphism), and $\Psi^+(X) \cong X'$ naturally.

\subsection{Universality}

We have already shown that $(\fD, \Phi, X, i)$ is a splitting category. We now show that it is universal. Thus suppose that $(\fT, \Psi, Y, j)$ is a second splitting category for $A$. We must construct a tensor functor $\Theta \colon \fD \to \fT$ and isomorphisms as in Definition~\ref{defn:usplit}.

For notational consistency, put $\fC'=\fT$. Let $A'=\Psi(A)$, and let $\Phi' \colon \fC' \to \fD'$ be the universal splitting category for $A'$. We have an induced tensor functor $\Psi^+ \colon \fD \to \fD'$ (see \S \ref{ss:split-func}). Since $A'$ is S-split, $\fD'$ is the category of $\bZ$-graded objects in $\fC'$ (see \S \ref{ss:split-case}). Thus the functor $\Phi'$ splits, that is, we have a tensor functor $\Pi \colon \fD' \to \fC'$ such that $\Pi \circ \Phi' = \id$. We define $\Theta=\Pi \circ \Psi^+$. The following diagram summarizes the situation:
\begin{displaymath}
\xymatrix@C=5em{
\fD \ar[r]^{\Psi^+} \ar@{..>}[rd]^{\Theta} & \fD' \ar@<-3pt>[d]_{\Pi} \\
\fC \ar[u]^{\Phi} \ar[r]^{\Psi} & \fC' \ar@<-3pt>[u]_{\Phi'} }
\end{displaymath}
It is clear that the requisite isomorphisms exist.

We must also show that $\Theta$ (together with the auxiliary isomorphisms) are unique up to isomorphism. For this, we note that every object of $\fD$ is a summand of an object of the form $M \otimes X_n$ with $M \in \fC$ and $n \in \bZ$. Thus a tensor functor out of $\fD$ is determined by its restriction to $\fC$ and where it sends $X$. This leads to the requisite uniqueness. We omit the details.

\subsection{Completion of proof}

At this point, we have shown that $(\fD, \Phi, X, i)$ is a universal splitting category and that statements~(a) and~(b) of Theorem~\ref{thm:split} hold. We now verify statement~(c). In what follows, we use that $\fD_n$ is equivalent to the category of relatively projective $A_n$-modules (Lemma~\ref{lem:D-equiv}).
\begin{itemize}
\item Suppose that $\fC$ is rigid. If $M$ is an object of $\fC$ then $M$ is a rigid object of $\fC=\fD_0$. Since $X_n$ is also a rigid object of $\fD$, it follows that $M \otimes X_n$ is rigid. Since every (homogeneous) object of $\fD$ is a summand of one of these objects, it follows that $\fD$ is rigid.
\item If $\fC$ is $\Hom$-finite then the category of $A_n$-modules is also $\Hom$-finite, and so $\fD$ is $\Hom$-finite.
\item Since $A$ is separable, so is $A_n$. Thus every $A_n$-module in $\fC$ is relatively projective (Proposition~\ref{prop:separable}), and so $\fD_n$ is equivalent to the category of all $A_n$-modules in $\fC$, which is abelian. Thus $\fD$ is abelian. 
\item If $\fC$ is semi-simple then the category of $A_n$-modules is semi-simple (Corollary~\ref{cor:azumss}), and coincides with the category of relatively projective modules, and so $\fD$ is semi-simple.
\item If $\fC$ is pre-Tannakian then so is $\fD$, by the above results.
\end{itemize}

\section{Splitting categories: complements} \label{s:split2}

\subsection{An example}

Let $\fC$ be the category of real vector spaces, and let $\bH$ be the usual Hamilton quaternions with basis 1, $i$, $j$, $k$. Define $\epsilon \colon \bH \to \bR$ by $\epsilon(1)=2$ and $\epsilon(x)=0$ if $x$ is one of $i$, $j$, or $k$. Also define $\pi \in \bH \otimes \bH$ by
\begin{displaymath}
\pi = \tfrac{1}{2} ( 1\otimes 1 - i \otimes i - j \otimes j - k \otimes k).
\end{displaymath}
Then $(\pi, \epsilon)$ is an S-structure on $\bH$. The S-algebra $\bH^{\otimes 2}$ is S-split, and so we get a $\bZ/2$-graded splitting category $\fD$ as in \S \ref{ss:msplit}. By the construction, $\fD_0=\fC$ is the category of real vector spaces, while $\fD_1$ is the category of (left) $\bH$-vector spaces.

Deligne pointed out an elegant model of the splitting category. Define $\fD'$ to be the category whose objects are $\bZ/2$-graded $\bC$-vector spaces $V$ equipped with a conjugate linear involution $\sigma \colon V \to V$ such that $\sigma^2(x)=(-1)^n \cdot x$ for $x \in V_n$. This category admits a natural symmetric tensor product, induced from the usual tensor product of graded vector spaces. The tensor category $\fD'$ is equivalent to the tensor category $\fD$ above.

\subsection{The Tannakian case}

Suppose $k$ is algebraically closed and $\fC=\Rep(G)$ is the category of finite dimensional representations of an affine group scheme $G$ over $k$. Let $A$ be an Azumaya algebra in $\fC$. Ignoring the $G$-action, $A$ is an Azumaya algebra in the category of finite dimensional vector spaces, and thus of the form $\End(V)$ for some vector space $V$. Since the automorphism group of the algebra $\End(V)$ is the algebraic group $\PGL(V)$, the action of $G$ on $A$ corresponds to a homomorphism $G \to \PGL(V)$ of algebraic groups, i.e., $V$ is a projective representation of $G$.

Let $\tilde{G}$ be the pull-back of $G$ along the natural map $\GL(V) \to \PGL(V)$; this is again an affine group scheme, and a central extension of $G$ by $\bG_m$. The universal splitting category of $A$ is $\fD=\Rep(\tilde{G})$; here we endow $A$ with the canonical S-structure on $\End(V)$. The functor $\fC \to \fD$ is simply pull-back along the projection $\tilde{G} \to G$. The $\bZ$-grading of $\fD$ is induced from the central $\bG_m$ in $\tilde{G}$. The space $V$ is naturally a linear representation of $\tilde{G}$, and we have an algebra isomorphism $A=\uEnd(V)$ in $\fD$.

\subsection{Connection to $\GL_t$} \label{ss:pgl}

Suppose $k$ has characteristic~0, and fix $t \in k$. We briefly recall the construction of the category $\uRep(\GL_t)$; see \cite{DeligneMilne} or \cite{Deligne} for additional details. To begin, we define a skeleton category $\uRep^0(\GL_t)$. The objects are formal symbols $V_{n,m}$ with $n,m \in \bN$, morphisms are linear combinations of walled Brauer diagrams, and composition uses the usual diagrammatic rule with parameter $t$. The category $\uRep(\GL_t)$ is defined to be the additive--Karoubi envelope of $\uRep^0(\GL_t)$. This category is $\bZ$-graded, with $V_{n,m}$ in degree $n-m$. We define $\uRep(\PGL_t)$ to be the degree~0 piece.

We let $V=V_{1,0}$, which is a degree~1 object of $\uRep(\GL_t)$ of categorical dimension $t$, and we let $A=\uEnd(V)$, which is an S-algbera of $\uRep(\PGL_t)$ of categorical degree $t$. Clearly, we can regard $\uRep(\GL_t)$ as a splitting category for $A$. In fact:

\begin{proposition}
The category $\uRep(\GL_t)$ is the universal splitting category of $A$.
\end{proposition}

\begin{proof}
This follows immediately from Corollary~\ref{cor:uni-char}.
\end{proof}

The categories $\uRep(\GL_t)$ and $\uRep(\PGL_t)$ have important mapping properties:

\begin{proposition}
Let $\fC$ be a Karoubian tensor category.
\begin{enumerate}
\item Giving a tensor functor $\uRep(\GL_t) \to \fC$ is equivalent to giving a rigid object of $\fC$ of categorical dimension $t$, via $\Phi \mapsto \Phi(V)$.
\item Giving a tensor functor $\uRep(\PGL_t) \to \fC$ is equivalent to giving an S-algebra of $\fC$ of categorical degree $t$, via $\Phi \mapsto \Phi(A)$.
\end{enumerate}
\end{proposition}

\begin{proof}
(a) is \cite[Proposition~10.3]{Deligne}. As far as we know, (b) does not appear explicitly in the literature, though it is likely known to experts; indeed, Noah Snyder has explained to us that it can be deduced from the results of \cite{Thurston}. It can also be deduced from our results, as follows. Let $B$ be an S-algebra in $\fC$ of categorical degree $t$. Let $\fD$ be the universal splitting category for $B$, and write $B=\uEnd(X)$ with $X \in \fD$. (We note that if $t=0$ then $\fD$ and $X$ are still provided by the constructions from \S \ref{s:split}.) By (a), there is a tensor functor $\uRep(\GL_t) \to \fD$ mapping $V$ to $X$. This restricts to a tensor functor $\Phi \colon \uRep(\PGL_t) \to \fC$ mapping $A$ to $B$. We omit the proof that $\Phi$ is uniquely determined up to isomorphism.
\end{proof}

Suppose now that we have an S-algebra $B$ in a Karoubian tensor category $\fC$ of categorical degree $t \ne 0$. As in the above proof, let $\fD$ be the universal splitting category of $B$, and write $B=\uEnd(X)$ where $X \in \fD$ has categorical dimension $t$. We have a diagram of tensor categories
\begin{displaymath}
\xymatrix{
\fC \ar[r] & \fD \\
\uRep(\PGL_t) \ar[u] \ar[r] & \uRep(\GL_t) \ar[u] }
\end{displaymath}
where the vertical maps come from the universal properties. The universality of $\fD$ exactly means that this is a push-out square. Thus the existence of universal splitting categories can be reformulated as saying that such push-outs always exist in the (1-truncation of the) 2-category of Karoubian tensor categories.

\begin{remark}
Harman has discussed how certain constructions with algebraic groups can be viewed as (co)limits in the 2-category of Tannakian categories. For instance, if $G$ is an algebraic group over a finite field $\bF$ then $\Rep(G(\bF))$ is the co-equalizer of the diagram $\Rep(G) \rightrightarrows \Rep(G)$, where the functors are the identity and Frobenius twist. The above description of the universal splitting category is similar in spirit.
\end{remark}

\subsection{The inseparable case} \label{ss:insep}

In \S \ref{s:split}, we assume that our S-algebra $A$ was separable. We now discuss what happens if we drop this assumption. The constructions of $\tilde{\fD}$, $\fD$, and the splitting structure on $\fD$ all go through; thus we still obtain a $\bZ$-graded extension of $\fC$ in which $A$ splits. Some issues start to arise \S \ref{ss:split-case}. Here, $A=\uEnd(V)$ is assumed to be split, and we made use of the Morita equivalence between $\fC$ and $\Mod_A$. This is likely not true in all tensor categories as we have defined them. However, this will still go through in a fairly general setting, e.g., if $\fC$ is abelian. This means that $\Phi \colon \fC \to \fD$ will still satisfy a universal property, but only among splitting categories for which this kind of Morita theory is valid.

There is one other issue in the inseparable case: it may not be the case that all $A$-modules are relatively projective, and so the category $\fD$ may fail to be abelian even if $\fC$ is abelian. To remedy this, one can redefine $\tilde{\fD}_{n,m}$ to be the category of all $A_{n,m}$-modules. This will yield an abelian tensor category in which $A$ splits. However, some other issues now arise. For instance, it is less clear if this version of $\fD$ inherits rigidity from $\fC$. We have not investigated these issues in detail.

\subsection{Change of S-structure}

Let $\fC$ be a tensor category, let $A$ be an S-algebra in $\fC$ with S-structure $(\pi, \epsilon)$, and let $A'$ be the algebra $A$ endowed with the S-structure $(-\pi, -\epsilon)$. Let $\fD$ and $\fD'$ be the universal splitting categories of $A$ and $A'$. The construction from \S \ref{s:split} shows that $\fD$ and $\fD'$ have the same underlying $k$-linear monoidal category; it is only the symmetry of the tensor product that differs. The two symmetries are related by the familiar super sign rule. Precisely, let $\tau$ be the symmetry on $\fD$, and $\tau'$ the one on $\fD'$. If $M$ and $N$ are objects of degrees $m$ and $n$ then $\tau'_{M,N} = (-1)^{nm} \tau_{M,N}$. This follows from a close examination of the construction.

\subsection{$\bZ/m$-gradings} \label{ss:msplit}

Let $A$ and $\fC$ be as in \S \ref{s:split}. Suppose we have an S-isomorphism $A^{\otimes m} = \uEnd(V)$ for some rigid object $V$ of $\fC$. We can then consider the following variant of the notion of splitting category:

\begin{definition}
An \defn{$m$-splitting category} for $A$ is a tuple $(\fD, \Phi, X, i, j)$ where $(\fD, \Phi, X, i)$ is a splitting category for $A$ and $j \colon \Phi(V) \to X^{\otimes m}$ is an isomorphism.
\end{definition}

We now explain how to construct a universal $m$-splitting category. Start with the universal splitting category $(\fD', \Phi', X', i')$ for $A$. Let $V_{-m} \in \fD_{-m}$ be $V$ regarded as an $A_{-m}$-module, and let $V_m \in \fD_m$ be $V^*$ regarded as an $A_m$-module. The same argument as used in Lemma~\ref{lem:Vinvert} shows that $V_{\pm m}$ is invertible with sign $+1$. Let $\fD=\fD'\{V_m\}$ be the trivialization of $V$ as in \S \ref{ss:triv1}. There is a natural functor $\fD' \to \fD$. We thus obtain a tensor functor $\Phi \colon \fC \to \fD$, and the pair $(X, i)$. There is a natural isomorphism $(X')^{\otimes m} \cong V_m \otimes V$ in $\fD'$ (where $V \in \fD'_0$), which yields the isomorphism $j \colon \Phi(V) \to X^{\otimes m}$ in $\fD'$. Universality is proven similarly to the previous case. The construction shows that $\fD$ is $\bZ/m$-graded, with $\fD_0=\fC$.

\subsection{Supersymmetric categories}

We have concentrated on Azumaya algebras with trivial $\eta$ invariant. We now make some comments on the $\eta \ne 1$ case. Let $\fD$ be a $k$-linear category equipped with an invertible object $L$ of sign $-1$ satisfying $L^{\otimes 2}=\bone$. We can enrich $\fD$ in super vector spaces by defining a degree~1 map $X \to Y$ to be a map $X \to L \otimes Y$ in $\fD$. Suppose now that we also have a decomposition $\fD=\bigoplus_{n \in \bZ} \fD_n$, and $\fD$ carries a monoidal structure $\otimes$. A \defn{supersymmetry} on $\otimes$ is a functorial isomorphism $X \otimes Y \to Y \otimes X$ that has degree $nm$ when $X \in \fD_n$ and $Y \in \fD_m$ such that a number of axioms hold (most notably a version of the hexagon axiom). There a number of subtleties, and we refer to \cite{supermon} for details.

Suppose $\fD$ carries a supersymmetry. Let $V$ be a rigid object of degree~1, and put $A=\uEnd(V)$. Note that $A$ lives in $\fC=\fD_0$, which is an ordinary (symmetric) tensor category. We have a natural isomorphism
\begin{displaymath}
A \otimes A \to L \otimes A \otimes A
\end{displaymath}
by applying the supersymmetry to the two $V$ factors, which shows that $\eta(A)=L$.

One can carry out our construction of splitting categories in this setting, i.e., given an Azumaya algebra $A$ in a (symmetric) tensor category $\fC$ with $\eta(A)$ of sign $-1$, one can construct a $\bZ$-graded supersymmetric extension $\fD$ of $\fC$ where $A$ splits.

\section{Interpolation via ultraproducts} \label{s:ultra}

\subsection{Set-up} \label{ss:ultra-setup}

Fix a finite field $\bF$ of odd cardinality $q$, and an infinite index set $I$ equipped with a non-principal ultrafilter. For $i \in I$, fix the following data
\begin{description}[align=right,labelwidth=1.5cm,leftmargin=!]
\item [ $E_i$ ] a symplectic vector space over $\bF$ of finite dimension $d_i$
\item [ $G_i$ ] the symplectic group $\Sp(E_i)$
\item [ $k_i$ ] an algebraically closed field in which $q \ne 0$
\item [ $\fC_i$ ] the category of finite dimensional representations of $G_i$ over $k_i$
\item [ $V_i$ ] the permutation representation $k_i[E_i]$ of $G_i$, an object of $\fC_i$
\end{description}
Let $k$ be the ultraproduct of the $k_i$'s. Let $t_i \in k_i$ be the image of $q^{d_i}$, and let $t \in k$ be the ultralimit of the $t_i$'s. We make the following assumptions:
\begin{enumerate}
\item $k$ has characteristic~0
\item the ultralimit of $\{d_i\}$ (in the ultrapower of $\bN$) is not an integer.
\end{enumerate}
If each $k_i$ has characteristic~0 then $t$ is transcendental over $\bQ$. By carefully choosing the $k_i$ and $d_i$, we can realize any non-zero algebraic number as $t$. We say that $t$ is \defn{singular} if it is an even power of $q$.

Let $\tilde{\fC}$ be the ultraproduct of the $\fC_i$, which is a $k$-linear abelian tensor category. The sequence $\{V_i\}$ defines a rigid object $V$ of $\tilde{\fC}$ of categorical dimension $t$. We let $\fC$ be the subcategory of $\tilde{\fC}$ generated by $V$, i.e., $\fC$ is the full subcategory spanned by objects which occur as a subquotient of a finite direct sum of tensor powers of $V$. If the $k_i$'s have characteristic~0, it is easy to see that $\fC$ is semi-simple and pre-Tannakian. In general, $\fC$ is pre-Tannakian, and semi-simple if $t$ is a non-singular parameter. These statements are not obvious, and do not yet appear in the literature, though see \cite[\S 8]{Kriz} for related results; in any case, they will not be essential to our discussion. The category $\fC$ is one model for $\uRep(\Sp_t(\bF))$.

\begin{remark}
It is not clear from the definition that $\fC$ depends only on $t$. The oligomorphic approach clarifies this issue.
\end{remark}

\subsection{Twisted group algebras} \label{ss:SpS}

Fix a non-trivial additive character $\psi \colon \bF \to k^{\times}$. We let $A_i$ be the twisted group algebra of $E_i$. The vector space underlying $A_i$ is simply $V_i$. The product is defined by
\begin{displaymath}
[x] \cdot [y] = \psi(\langle x, y \rangle) [x+y].
\end{displaymath}
Let $t_i^{1/2} \in k_i$ be the square root $q^{d_i/2}$ of $t_i$. Define $\epsilon_i \colon A_i \to k$ and $\pi_i \in A_i \otimes A_i$ by
\begin{displaymath}
\epsilon_i \big( \sum c_x [x] \big) = t_i^{1/2} \cdot c_0, \qquad
\pi_i = t_i^{-1/2} \cdot \sum_{x+y=0} [x] \otimes [y].
\end{displaymath}
We now have:

\begin{proposition} \label{prop:tgas}
With the above data, $A_i$ is an S-algebra of formal degree $t_i^{1/2}$.
\end{proposition}

\begin{proof}
We identify $A_i \otimes A_i$ with $k_i[E_i \oplus E_i]$ in what follows, and write $[x,y]$ in place of $[(x,y)]=[x] \otimes [y]$. With this notation, we have
\begin{displaymath}
\pi_i = t_i^{-1/2} \sum_{u \in E_i} [u,-u].
\end{displaymath}
Let $x,y \in E_i$. Then
\begin{align*}
\pi_i \cdot [x,y] \cdot \pi_i
&= t_i^{-1} \sum_{(u,v) \in E_i \oplus E_i} [u,-u] \cdot [x,y] \cdot [-v, v] \\
&= t_i^{-1} \sum_{(u,v) \in E_i \oplus E_i} \psi(-2 \langle u,v \rangle+\langle u+v, x-y \rangle) [x+u-v,y-u+v].
\end{align*}
Put $a=u+v$ and $b=u-v$. We then have
\begin{align*}
\pi \cdot [x,y] \cdot \pi
&= t_i^{-1} \sum_{b \in E_i} \sum_{a \in E_i} \psi(\langle a, b+x-y \rangle) [x+b,y-b].
\end{align*}
The inner sum vanishes unless $b=y-x$, in which case it is $t_i$. We thus find
\begin{displaymath}
\pi \ast [x,y] \ast \pi = [y,x],
\end{displaymath}
and so $\tau=\lambda_{\pi} \rho_{\pi}$.

We have
\begin{displaymath}
[x,y] \ast \pi
= t_i^{-1/2} \sum_{u \in E_i} [x,y] \cdot [u,-u]
= t_i^{-1/2} \sum_{u \in E_i} \psi(\langle x-y, u \rangle) [x+u,y-u]
\end{displaymath}
Now apply $\id \otimes \epsilon$. Only that $u=y$ term in the integral contributes, and so we find
\begin{displaymath}
(\id \otimes \epsilon) \rho_{\pi} [x,y] = \psi(\langle x-y, y \rangle) [x+y] = \psi(\langle x, y \rangle) [x+y]
\end{displaymath}
This coincides with $\mu([x] \otimes [y])$, and so we have verified that $(\pi_i, \epsilon_i)$ is an S-structure. As $\epsilon_i(1)=t_i^{1/2}$, the result follows.
\end{proof}

We now let $(A, \pi, \epsilon)$ be the ultralimit of the $(A_i, \pi_i, \epsilon_i)$. This is an S-algebra in $\fC$ of categorical degree $t^{1/2}$. It is therefore also an Azumaya algebra in $\fC$ with $\tilde{\eta}(A)=1$ by Theorem~\ref{thm:azu-T}.

\begin{remark}
If $\bF$ has characteristic~2 then the algebra $A_i$ would be commutative, and thus could not admit a S-structure (the symmetry of $A_i \otimes A_i$ could not be inner). In the above proof, we used the characteristic $\ne 2$ hypothesis when we changed coordinates from $(u,v)$ to $(a,b)$, since the inverse transformation involves dividing by~2.
\end{remark}

\subsection{Splittings}

In what follows, we say that an anti-involution $\sigma$ of an S-algebra is \defn{admissible} if it preserves $\epsilon$, i.e., $\epsilon = \epsilon \circ \sigma$.

\begin{proposition} \label{prop:Asplit}
We have the following:
\begin{enumerate}
\item If $-1$ is a square in $\bF$ then $A_i$ admits an admissible anti-involution.
\item It is always the case that $A_i^{\otimes 2}$ admits an admissible anti-involution.
\end{enumerate}
\end{proposition}

\begin{proof}
(a) Let $i=\sqrt{-1}$, and let $\sigma \colon A_i \to A_i$ be the map $\sigma([x])=[ix]$. We have
\begin{displaymath}
\sigma([x] \cdot [y]) = \sigma(\psi(\langle x, y \rangle) [x+y]) = \psi(\langle x, y \rangle) [i(x+y)]
\end{displaymath}
and
\begin{displaymath}
\sigma([y]) \cdot \sigma([x]) = [iy] \cdot [ix] = \psi(\langle iy, ix \rangle) [i(x+y)] = \psi(\langle x, y\rangle) [i(x+y)],
\end{displaymath}
and so $\sigma$ is an anti-involution. It is clearly admisisble.

(b) Let $u,v \in \bF^2$ be vectors such that
\begin{displaymath}
u \cdot u=-1, \qquad u \cdot v = 0, \qquad v \cdot v=-1,
\end{displaymath}
where here $\cdot$ is the standard dot product. The existence of such $u$ and $v$ follows from the theory of quadratic forms over finite fields. Define\footnote{The idea to use this map came from \cite{DeligneLetter1}.}
\begin{displaymath}
T \colon E_i^{\oplus 2} \to E_i^{\oplus 2}, \qquad T(x,y)=(u_1 x+v_1 y, u_2 x+v_2 y).
\end{displaymath}
One readily verifies
\begin{displaymath}
\langle T(x,y), T(x',y') \rangle = - \langle (x,y), (x',y') \rangle.
\end{displaymath}
We now define $\sigma \colon A^{\otimes 2} \to A^{\otimes 2}$ by $\sigma([x,y])=[T(x,y)]$. The remainder of the proof is similar to case (a).
\end{proof}

\begin{corollary} \label{cor:Asplit}
We have the following:
\begin{enumerate}
\item If $-1$ is a square in $\bF$ then $A^{\otimes 2}$ is S-split.
\item It is always the case that $A^{\otimes 4}$ is S-split.
\end{enumerate}
\end{corollary}

\begin{proof}
(a) Applying the proposition and taking an ultralimit, we see that $A$ carries an admissible anti-involution. Thus $A^{\otimes 2}$ is S-split by Proposition~\ref{prop:anti-inv}. Case (b) is similar.
\end{proof}

\begin{remark}
In fact, $A_i$ itslef is S-split, provided $k_i$ has characteristic~0: it is the endomorphism algebra of the oscillator representation. However, $A$ is not split (Corollary~\ref{cor:order4}), at least for generic $t$.
\end{remark}

\subsection{The main construction} \label{ss:maincon}

We have produced an S-algebra $A$ in $\fC$. It has categorical degree $t^{1/2}$, and thus categorical dimension $t$, and is hence separable (Proposition~\ref{prop:separable}). Theorem~\ref{thm:split} therefore applies, and produces a $\bZ$-graded universal splitting category. In fact, this comes from a $\bZ/2$-graded category if $-1$ is a square in $\bF$, or a $\bZ/4$-graded category if not, by Corollary~\ref{cor:Asplit} and \S \ref{ss:msplit}. This category is pre-Tannakian, and semi-simple for non-singular $t$ (if we grant these statements about $\fC$). The construction of these tensor categories is the main point of this paper.

\subsection{Comparison with Kriz's category} \label{ss:kriz}

We assume in \S \ref{ss:kriz} that $t$ is transcendental (and take $k_i=\bC$ for simplicity), and that $-1$ is not a square in $\bF$. This implies that $\fC$ is semi-simple. See Remark~\ref{rmk:comments} for comments on other cases.

Let $W_i$ be the oscillator representation (or Weil representation) of $G_i$ over $k_i$; see \cite[\S 2]{Kriz}. This representation has dimension $q^{d_i/2}$, and there is a natural identification $A_i=\uEnd(W_i)$ of algebras in $\fC_i$. We also have an isomorphism
\begin{equation} \label{eq:weil}
W_i^{\otimes 4} \cong A_i^{\otimes 2}
\end{equation}
by \cite[Lemma~15]{Kriz}. Let $W \in \tilde{\fC}$ be the ultralimit of the $W_i$'s. For $a \in \bZ$, define $\fD'_a$ to be the full subcategory of $\tilde{\fC}$ spanned by objects that occur as a subquotient of a finite direct sum of objects of the form $W^{\otimes a} \otimes V^{\otimes r}$, with $r$ arbitrary. This is an abelian category, and we clearly have $\fD'_a \otimes \fD'_b \subset \fD'_{a+b}$. Moreover, \eqref{eq:weil} implies $\fD'_a=\fD'_b$ if $a \equiv b \pmod{4}$. Deligne \cite{DeligneLetter1} proved the following result; see also \cite[\S 8]{Kriz}:

\begin{theorem} \label{thm:mod4}
If $a,b \in \bZ/4$ are distinct, $X \in \fD'_a$, and $Y \in \fD'_b$, then $\Hom_{\tilde{\fC}}(X,Y)=0$.
\end{theorem}

We thus see that $\fD'$ is a $\bZ/4$-graded tensor category with $\fD'_0=\fC$. The ultralimit of the isomorphisms $A_i=\uEnd(W_i)$ shows that $A=\uEnd(W)$ in $\fD'$. Thus $\fD'$ is naturally a 4-splitting category for $A$ as in \S \ref{ss:msplit}. The $\bZ/m$-version of Corollary~\ref{cor:uni-char} now shows that $\fD'$ is equivalent to the universal 4-splitting category $\fD$ constructed in \S \ref{ss:maincon}. Kriz proves that her category is equivalent to $\fD'$, so it is thus equivalent to $\fD$ as well.

\begin{remark} \label{rmk:comments}
If $-1$ is a square then Deligne shows that $\fD$ is $\bZ/2$-graded, and the above arguments show that it matches the $\bZ/2$-graded version of $\fD$ from \S \ref{ss:maincon}. In fact, Deligne constructs a $\bZ/2 \times \bZ/2$-graded category by making use of a second oscillator representation (obtained by using a different choice of $\psi$). This can be carried out using our constructions as well. Deligne's proof of Theorem~\ref{thm:mod4} uses the characteristic~0 representation theory of the $G_i$'s. To extend this theorem to general $t$'s would require analogous results in positive characteristic.
\end{remark}

\subsection{Non-trivial Brauer classes}

We maintain the assumptions that $t$ is transcendental and that $-1$ is not a square in $\bF$. We now show:

\begin{theorem}
The Azumaya algebra $A^{\otimes 2}$ is not split in $\fC$.
\end{theorem}

\begin{proof}
Suppose we have a splitting $A^{\otimes 2} \cong \uEnd(X)$ for some rigid object $X$ of $\fC$. We thus have $A_i^{\otimes 2} \cong \uEnd(X_i)$ for all $i$ belonging to some set $I'$ in our ultrafilter. For all $i$, we have a matural isomorphism $A_i = \uEnd(W_i)$, and so $A_i^{\otimes 2} = \uEnd(W_i^{\otimes 2})$. The following lemma, combined with the fact that $G_i$ has trivial abelianzation, implies that $X_i \cong W_i^{\otimes 2}$ in $\fC_i$, and so $X \cong W^{\otimes 2}$ in $\fD'$. But this contradicts Theorem~\ref{thm:mod4}, since $X$ is in degree~0 and $W$ is in degree~1.
\end{proof}

\begin{lemma}
Let $\Gamma$ be a group and let $V$ and $W$ be finite dimensional complex representations such that $\End(V) \cong \End(W)$ as $\Gamma$-algebras. Then there is a one-dimensional representation $L$ of $\Gamma$ such that $V \cong L \otimes W$.
\end{lemma}

\begin{proof}
Let $\rho \colon \Gamma \to \GL(V)$ and $\sigma \colon \Gamma \to \GL(W)$ be the representations, and let $\phi \colon \End(V) \to \End(W)$ be the given isomorphism. By Noether--Skolem, $\phi$ is induced by a linear map $\psi \colon V \to W$, i.e., we have $\phi(\alpha)=\psi \alpha \psi^{-1}$ for $\alpha \in \End(V)$. Since $\phi$ is $\Gamma$-equivariant, we find
\begin{displaymath}
\sigma(g) \psi \alpha \psi^{-1} \sigma(g)^{-1} = \sigma(g) \phi(\alpha) \sigma(g)^{-1} = \phi(\rho(g) \alpha \rho(g)^{-1}) = \psi \rho(g) \alpha \rho(g)^{-1} \psi^{-1}.
\end{displaymath}
Thus $\psi^{-1} \sigma(g)^{-1}\psi \rho(g)$ commutes with $\alpha$ for all $\alpha$, and is therefore a scalar $\lambda(g)$. We thus have $\rho(g) = \lambda(g) \psi^{-1} \sigma(g) \psi$. One easily sees that $\lambda$ is a group homomorphism, and so the result follows.
\end{proof}

Let $\Br(\fC)$ denote the Brauer group of $\fC$, as defined in \cite{Pareigis4}.

\begin{corollary} \label{cor:order4}
The class $[A]$ in $\Br(\fC)$ generates a cyclic subgroup of order~4.
\end{corollary}

\begin{proof}
The theorem shows that $[A]$ is not 2-torsion, while Corollary~\ref{cor:Asplit} shows that $[A]$ is 4-torsion.
\end{proof}

\begin{remark}
As in Remark~\ref{rmk:comments}, if $-1$ is a square then we can construct a subgroup of $\Br(\fC)$ isomorphic to the Klein 4-group. We also expect that these results will extend to more general values of $t$ (at least all non-singular values, but perhaps all $t \ne 0$).
\end{remark}

\begin{remark}
The categorical dimensions of objects of $\fC$ belong to $\bQ(t) \subset k$, though this statement is not yet in the literature. Granted this, it is obvious that $A$ is not split: indeed, if $A \cong \uEnd(V)$ then taking categorical dimensions we would find $t=(\udim{V})^2$, a contradiction. The fact that $A^{\otimes 2}$ is non-split does not seem so immediate.
\end{remark}

\section{Interpolation via oligomorphic groups} \label{s:oligo}

\subsection{Background}

We very briefly recall some of the theory from \cite{repst}. We refer to \cite[\S 2]{line} for a more substantial summary. We let $k$ be an algebraically closed field of characteristic~0 throughout.

Let$(G, \Omega)$ be an oligomorphic group. This means that $G$ is a permutation group of $\Omega$, and $G$ has finitely many orbits on $\Omega^n$ for all $n \ge 0$. For a finite subset $A$ of $\Omega$, let $G(A)$ be the subgroup of $G$ fixing each element of $A$. These subgroups form a neighborhood basis for a topology on $G$. We say that an action of $G$ on a set $X$ is \defn{smooth} if each stabilizer is an open subgroup, meaning it contains some $G(A)$, and \defn{finitary} if $G$ has finitely many orbis on $X$. We use the term ``$G$-set'' to mean ``set equipped with a smooth and finitary action of $G$.'' See \cite[\S 2]{repst} for more details.

Let $X$ be a $G$-set. Let $\cC(X)$ denote the \defn{Schwartz space}, i.e., the space of smooth functions $\phi \colon X \to k$, meaning $\phi$ is stable under left translation by some open subgroup. One of the key ideas introduced in \cite{repst} was a certain notion of a \defn{measure} on $G$ (\cite[Definition~3.1]{repst}); this assigns a value in $k$ to each $G$-set (in fact, to each $\hat{G}$-set). Fixing such a measure $\mu$, we can define the integral
\begin{displaymath}
\int_X \phi(x) dx
\end{displaymath}
for $\phi \in \cC(X)$. This integral is in fact a finite sum. See \cite[\S 4]{repst} for details.

We define a tensor category $\uPerm(G, \mu)$ as follows. The objects are the Schwartz spaces $\cC(X)$, where $X$ is a $G$-set. A morphism $\cC(X) \to \cC(Y)$ is given by a $G$-invariant $Y \times X$ matrix, meaning a $G$-invariant function $Y \times X \to k$. Such a matrix defines an actual linear function $\cC(X) \to \cC(Y)$, and composition is just function composition. The direct sum $\oplus$ and tensor product $\uotimes$ are given on objects by
\begin{displaymath}
\cC(X) \oplus \cC(Y) = \cC(X \amalg Y), \qquad
\cC(X) \uotimes \cC(Y) = \cC(X \times Y).
\end{displaymath}
See \cite[\S 8]{repst} for details. We note that $\oplus$ is the usual direct sum on the underlying vector spaces, while $\uotimes$ is not the usual tensor product, which is why we decorate the symbol. The tensor category $\uPerm(G, \mu)$ is rigid, and the objects $\cC(X)$ are self-dual. It is an important and difficult problem to determine when $\uPerm(G, \mu)$ admits a pre-Tannakian envelope; see \cite[Theorem~13.2]{repst} for one result in this direction.

\subsection{Symplectic spaces} \label{ss:sympG}

Fix a finite field $\bF$ of odd characteristic. We define a \defn{symplectic $G$-space} to be $\bF$-vector space $V$ equipped with a linear action of $G$ and an alternating form $\langle, \rangle$ such that:
\begin{itemize}
\item $V$ is a $G$-set, i.e., the action is finitary and smooth.
\item The form is $G$-invariant, i.e., $\langle gv, gw \rangle = \langle v, w \rangle$ for $v,w \in V$ and $g \in G$.
\item The form is non-degenerate in the sense that it has no kernel, i.e., if $\langle v, w \rangle=0$ for all $w$ then $v=0$.
\end{itemize}
In the main case of interest, $V$ will have dimension $\aleph_0$, and so the non-degeneracy condition does not imply that the map $V \to V^*$ is an isomorphism.

\subsection{Twisted convolution algebras} \label{ss:oligo-alg}

Let $V$ be a $G$-symplectic space, and fix a non-trivial additive character $\psi \colon \bF \to k^{\times}$. We define an algebra $A=A(V)$ as follows. The underlying vector space is $\cC(V)$, the space of Schwartz functions on $V$. The product $\ast$ is the following twisted convolution:
\begin{displaymath}
(\phi_1 \ast \phi_2)(x) = \int_V \psi(\langle y, x-y \rangle) \phi_1(y) \phi_2(x-y) dy
\end{displaymath}
One easily sees (using Fubini's theorem \cite[Corollary~4.9]{repst}) that this product is associative, and that the point mass $\delta_0$ at the origin is the unit. As a special case, we find
\begin{displaymath}
\delta_x \ast \delta_y = \psi(\langle x, y \rangle) \delta_{x+y},
\end{displaymath}
for $x,y \in V$; we note that maps of Schwartz spaces are determined by where they send point masses, so the above formula uniquely characterizes the product. It is clear that the product is induced by a $(V \times V) \times V$ matrix, and so the multiplication map is a map in the category $\uPerm(G, \mu)$. We thus see that $A$ is an algebra object in the category $\uPerm(G, \mu)$.

We now assume $\mu(V) \ne 0$, and let $\gamma$ be a chosen square root of $\mu(V)$. Define $\pi \in \cC(V \oplus V)$ and $\epsilon \colon \cC(V) \to \bone$ by
\begin{displaymath}
\pi(x,y)=\gamma^{-1} \cdot \delta_0(x+y), \qquad \epsilon(\phi)=\gamma \cdot \phi(0).
\end{displaymath}
One shows that $(\pi, \epsilon)$ is an S-structure on $A$, by essentially the same argument used to prove Proposition~\ref{prop:tgas}. Similarly, the argument from Proposition~\ref{prop:Asplit} shows that $A^{\otimes 2}$ is S-split if $-1$ is a square in $\bF$, and $A^{\otimes 4}$ is always S-split.

\subsection{The symplectic group}

We now specialize to the infinite symplectic group. Let $\bV_i$ be a two-dimensional symplectic vector space over $\bF$ with symplectic basis $e_i, f_i$, and let $\bV=\bigoplus_{i \ge 1} \bV_i$, equipped with the natural symplectic form. We let $G=\Sp(\bV)$ be the group of all linear automorphisms of $\bV$ preserving the form. This group is oligomorphic via its action on $\bV$, and $\bV$ is a symplectic $G$-space as defined in \S \ref{ss:sympG}.

The group $G$ carries a 1-parameter family of measures $\mu_t$. Fix a choice of parameter $t \in k \setminus \{0\}$. The category $\uPerm(G, \mu_t)$ admits a pre-Tannakian abelian envelope $\fC=\uRep(G, \mu_t)$, which is semi-simple if $t$ is non-singular (as in \S \ref{ss:ultra-setup}). This category is equivalent to the ultraproduct category from \S \ref{ss:ultra-setup}. These statements do not yet appear in the literature, but can be proven similarly to the $\GL(\bV)$ case, which is treated in detail in \cite[\S 15]{repst}.

Applying the construction from \S \ref{ss:oligo-alg}, we obtain an S-algebra $A=A(\bV)$ in $\fC$. We now form the universal splitting category of $A$ (or its $\bZ/2$ or $\bZ/4$ variant) to obtain a category that is once again equivalent to Kriz's.

\end{document}